\documentclass{article}

\usepackage[margin=1in,a4paper]{geometry}
\usepackage{amssymb,amsmath,,wasysym,amsthm} 
\usepackage{enumitem}
\usepackage{xcolor}
\usepackage{float}
\usepackage[colorlinks,linkcolor=blue,citecolor=blue,pagebackref,hypertexnames=false, breaklinks]{hyperref}
\usepackage{circuitikz}

\def\Xint#1{\mathchoice
    {\XXint\displaystyle\textstyle{#1}}%
    {\XXint\textstyle\scriptstyle{#1}}%
    {\XXint\scriptstyle\scriptscriptstyle{#1}}%
    {\XXint\scriptscriptstyle\scriptscriptstyle{#1}}%
    \!\int}
    \def\XXint#1#2#3{{\setbox0=\hbox{$#1{#2#3}{\int}$}
    \vcenter{\hbox{$#2#3$}}\kern-.5\wd0}}
    \def\fint{\Xint-}

\def\vint{\mathop{\mathchoice%
          {\setbox0\hbox{$\displaystyle\intop$}\kern 0.22\wd0%
           \vcenter{\hrule width 0.6\wd0}\kern -0.82\wd0}%
          {\setbox0\hbox{$\textstyle\intop$}\kern 0.2\wd0%
           \vcenter{\hrule width 0.6\wd0}\kern -0.8\wd0}%
          {\setbox0\hbox{$\scriptstyle\intop$}\kern 0.2\wd0%
           \vcenter{\hrule width 0.6\wd0}\kern -0.8\wd0}%
          {\setbox0\hbox{$\scriptscriptstyle\intop$}\kern 0.2\wd0%
           \vcenter{\hrule width 0.6\wd0}\kern -0.8\wd0}}%
          \mathopen{}\int}

\def\Xint#1{\mathchoice
    {\XXint\displaystyle\textstyle{#1}}%
    {\XXint\textstyle\scriptstyle{#1}}%
    {\XXint\scriptstyle\scriptscriptstyle{#1}}%
    {\XXint\scriptscriptstyle\scriptscriptstyle{#1}}%
    \!\int}
    \def\XXint#1#2#3{{\setbox0=\hbox{$#1{#2#3}{\int}$}
    \vcenter{\hbox{$#2#3$}}\kern-.5\wd0}}
    \def\fint{\Xint-}

\newcommand{\pip}{\varphi}

\newcommand{\eps}{{\varepsilon}\, }

\newcommand{\supp}{{\rm supp}\, }
\newcommand{\dom}{{\rm dom}\, }

\theoremstyle{plain}
\newtheorem{thm}{Theorem}[section]
\newtheorem{theorem}{Theorem}[section]
\newtheorem{lem}[thm]{Lemma}
\newtheorem{cor}[thm]{Corollary}
\newtheorem{prop}[thm]{Proposition}

\newtheorem{example}[thm]{Example}
\newtheorem{assump}[thm]{Assumption}

\theoremstyle{definition}
\newtheorem{defn}[thm]{Definition}

\theoremstyle{remark}
\newtheorem{remark}[thm]{Remark}
\newcommand{\bremark}{\begin{remark} \em}
\newcommand{\eremark}{\end{remark} }
\usepackage{color}

\hbadness=\maxdimen

\begin{document}


\title{Dirichlet forms on metric measure spaces as Mosco limits of  Korevaar-Schoen energies}

\author{Patricia Alonso Ruiz\footnote{Partly supported by the NSF grant DMS 1951577}, Fabrice Baudoin\footnote{Partly supported by the NSF grant DMS~1901315.}}

\maketitle

\begin{abstract}
This paper establishes sufficient general conditions for the existence of Mosco limits of Korevaar-Schoen $L^2$ energies, first in the context of Cheeger spaces and then in the context of fractal-like spaces with walk dimension greater than 2. Among other ingredients, a new Rellich-Kondrachov type theorem for Korevaar-Schoen-Sobolev spaces is of independent interest.
\end{abstract}

\bigskip


\medskip


\tableofcontents

\section{Introduction}
The important question of how to construct a diffusion process intrinsic to a given ``generic'' metric measure space was stressed by the influential work of Sturm~\cite{Stu98a,Stu98b} and has 
become increasingly relevant in the study of analytic, geometric and probabilistic aspects of highly irregular, non-smooth spaces. To construct a ``naturally intrinsic'' Brownian motion on fractals~\cite{Kus89,BB89,Kig89,KS05}, an initial key observation was to rephrase the construction of Brownian motion on $\mathbb{R}^n$ in terms of Dirichlet forms: instead of finding the limit of suitably rescaled random walks one could alternatively consider the convergence of sequences of non-local regular Dirichlet forms. 

\medskip

This latter approach was laid out by Kumagai and Sturm in the context of metric measure spaces in~\cite{KS05}. More precisely, for a locally compact doubling space $(X,d,\mu)$ they provided sufficient conditions to guarantee the \emph{$\Gamma$-convergence} of the non-local Dirichlet forms 
\begin{equation}\label{E:DF_intro}
\mathcal{E}^{r_n}(f,f):=\frac{1}{h(r_n)}\int_X\frac{1}{\mu(B(x,r_n))}\int_{B(x,r_n)}|f(x)-f(y)|^2d\mu(y)\,d\mu(x),\qquad f\in L^2(X,\mu),
\end{equation}
where $h(r_n)$ is a suitable  scale function  and $r_n\to 0$. The notion of $\Gamma$-convergence was introduced by de Giorgi and Franzoni in~\cite{dGF75} to study variational problems. Broadly speaking, given a sequence of functionals like $\{\mathcal{E}^{r_n}(\cdot,\cdot)\}_{n\geq 1}$ and a sequence of minimizers $\{f_n\}_{n\geq 1}$, one is interested in the question whether there is a meaningful limiting functional $\mathcal{E}$ and a limiting function $f$ that is also a minimizer of $\mathcal{E}$. The type of convergence that gives the limiting functional that property is known as $\Gamma$-convergence. 


\medskip

A related and stronger variational notion of convergence with further-reaching implications is the so-called \emph{Mosco convergence} introduced by Mosco in~\cite{Mos94}. Among others, the Mosco convergence of a sequence such as $\{\mathcal{E}^{r_n}(\cdot,\cdot)\}_{n\geq 1}$ can be characterized in terms of convergence of the semigroups associated with the Dirichlet forms. Especially relevant from a probabilistic point of view is the fact that Mosco convergence implies the weak convergence of the finite-dimensional distributions of the associated Hunt processes. This property makes it the ``correct'', or at least the most useful, notion of convergence for Dirichlet forms. Mosco convergence is now established as part of the main results in the present paper, Theorem~\ref{T:main} and Theorem~\ref{T:main_dw}, for general compact metric measure spaces. Passing from $\Gamma$-convergence to Mosco convergence is an involved matter and will require in particular developing a suitable Rellich-Kondrachov theorem of independent interest in the spirit of~\cite{HK00}, see Section~\ref{S:Rellich-Kondrachov}.

\medskip

In the case when $h(r)=r^{d_w}$ for some $d_w\geq 2$, the Dirichlet form~\eqref{E:DF_intro} resembles the energy functionals introduced by Korevaar and Schoen in their seminal paper~\cite{KS93}. The versatility and applicability of these functionals in the metric measure space setting, see e.g.~\cite{Bau22,KST04,MMS16,BV2,BV3}, has led to an increasing interest and body of work around them and will therefore be the focus of the present paper. Our main purpose is to characterize the existence of local and regular Dirichlet forms $(\mathcal{E},\dom\mathcal{E})$ as limits of Korevaar-Schoen energy functionals under minimal requirements on the underlying space $(X,d,\mu)$. 




The paper is organized as follows: After a brief description of the initial set up, Section~\ref{S:strictly_local} deals with the case $d_w=2$ and the characterization of strictly local and regular Dirichlet forms in terms of the validity of a 2-Poincar\'e inequality with respect to Lipschitz constants. The main result is stated in Theorem~\ref{T:main} and its proof subdivided in several steps, that correspond to different aspects of the limiting Dirichlet form. Each of these are treated in separate sections; Mosco convergence is investigated in Subsection~\ref{S:Mosco_convergence} and obtained for compact spaces. While not an a priori assumption, in the case $d_w=2$ the limiting Dirichlet form is strictly local and in particular both the original and the intrinsic distance are bi-lipschitz equivalent. 
It is worth noting that while working on this paper, through discussions with N. Shanmugalingam,  we became aware of a work in progress \cite{WIPNages} where, in the same setting with $d_w=2$, Dirichlet forms are constructed as Mosco limits of discrete energies associated with graph approximations of the metric measure space $(X,d,\mu)$.
Moving to the general case $d_w\geq 2$ in Section~\ref{S:non-strictly}, Subsection~\ref{S:Rellich-Kondrachov} is devoted to developing a Rellich-Kondrachov theorem of interest in its own. This compactness argument is applied in Subsection~\ref{SS:non-strictly} to prove the corresponding existence result, Theorem~\ref{T:main_dw}. The exposition finishes 
in Subsection~\ref{SS:converse_main} with a discussion of a converse to Theorem~\ref{T:main_dw}.

\

\textbf{Notation:} If $a$ are $b$ are two non negative functionals we will use the notation $a \simeq b$ to indicate that there exists a constant $C\ge 1$ such that $\frac{1}{C} a \le b \le C a$.

\section{Definitions and setup}

\subsection{Metric measure space}\label{MMSintro}
Throughout the paper we consider a locally compact, complete, metric measure space $(X,d,\mu)$ where $\mu$ is a Radon measure. Any open metric ball centered at $x\in X$ with radius $r>0$ will be denoted by
\[
B(x,r)= \{ y \in X, d(x,y)<r \}.
\] 
When convenient, for a ball $B:=B(x,r)$ and $\lambda>0$, we will denote by $\lambda B$ the ball $B(x,\lambda r)$. 

The measure $\mu$ will be assumed to be doubling and positive in the sense that there exists a constant $C>0$ such that for every $x \in X, r>0$,
\begin{equation}\label{A:VD}
0< \mu (B(x,2r)) \le C \mu(B(x,r)) <+\infty.\tag{$\rm VD$}
\end{equation}

From the doubling property of $\mu$, see e.g.~\cite[Lemma 8.1.13]{HKST15}, it follows that there exist constants $C > 0$ and $0<Q\ <\infty$ such that
\begin{equation}\label{eq:mass-bounds}
 \frac{\mu(B(x,R))}{\mu(B(x,r))}
 \le C\left(\frac{R}{r}\right)^Q
\end{equation}
for any $0<r\le R$ and $x\in X$. Another useful consequence of the doubling property is the availability of maximally separated $\eps$-coverings with the bounded overlap property and subordinated Lipschitz partitions of unity, see \cite[pp. 102-104]{HKST15}. 

\subsection{Dirichlet forms}
 Let $(\mathcal{E},\mathcal{F}:=\rm{dom}\,\mathcal{E})$ be a densely defined closed symmetric form on $L^2(X,\mu)$. A function $v\colon X\to\mathbb{R}$ is called a normal contraction of the function $u$ if for $\mu$-almost every $x,y \in X$,
\[
| v(x)-v(y)| \le |u(x) -u(y)| \qquad\text{and}\qquad |v(x)| \le |u(x)|.
\]
The form $(\mathcal{E},\mathcal{F})$ is called a Dirichlet form if it is Markovian, that is, it has the property that if $u \in \mathcal{F}$ and $v$ is a normal contraction of $u$ then $v \in \mathcal{F}$ and $\mathcal{E}(v,v) \le \mathcal{E} (u,u)$.

 Some basic properties of Dirichlet forms are collected in~\cite[Theorem 1.4.2]{FOT11}. In particular, we note that $\mathcal{F} \cap L^\infty(X,\mu)$ is an algebra and $\mathcal F$ is a Hilbert space equipped with the $\mathcal{E}_1$-norm defined as
\begin{equation}\label{e-E1}
\|f\|_{\mathcal{E}_1}:=\left( \| f \|_{L^2(X,\mu)}^2 + \mathcal{E}(f,f) \right)^{1/2}.
\end{equation} 

Denoting by $C_c(X)$ the space of continuous functions with compact support in $X$, we recall that a core for $(\mathcal{E},\mathcal{F})$ is a subset $\mathcal{C}\subseteq C_c(X) \cap \mathcal{F}$ which is dense in $C_c(X)$ in the supremum norm and dense in $\mathcal{F}$ in the $\mathcal{E}_1$-norm.

\begin{defn} 
A Dirichlet form $(\mathcal{E},\mathcal{F})$ on $L^2(X,\mu)$ is called regular if it admits a core.
\end{defn}

Given a regular Dirichlet form $(\mathcal{E},\mathcal{F})$, for every $u,v\in \mathcal F\cap L^{\infty}(X)$ one defines the energy measure $\Gamma (u,v)$ through the formula
 \[
\int_X\phi\, d\Gamma(u,v)=\frac{1}{2}[\mathcal{E}(\phi u,v)+\mathcal{E}(\phi v,u)-\mathcal{E}(\phi, uv)], \quad \phi\in \mathcal F \cap C_c(X).
\]
The latter definition of $\Gamma(u,v)$ can be extended to any $u,v\in \mathcal F$ by truncation, see e.g.~\cite[Theorem 4.3.11]{CF12}. According to Beurling and Deny~\cite{BD59}, $\Gamma(u,v)$ is a signed Radon measure for any $u,v\in \mathcal F$ and 
\[
\mathcal E(u,v)=\int_X d\Gamma(u,v).
\]
For $u \in \mathcal{F}$, $\Gamma(u,u)$ is called  the energy measure of $u$.

\medskip

We will be concerned with two main types of regular Dirichlet forms, \emph{strongly} and \emph{striclty} local. 

\begin{defn}(\cite[p.6]{FOT11}) 
A Dirichlet form $(\mathcal{E},\mathcal{F})$ is called  \emph{strongly local} if for any $u,v\in\mathcal{F}$ with compact supports such that $u$ is constant in a 
neighborhood of the support of $v$, it holds that $\mathcal{E}(u,v)=0$. 
\end{defn}

With respect to a regular Dirichlet form $(\mathcal{E},\mathcal{F})$ one can define the following \emph{intrinsic metric} $d_{\mathcal{E}}$ on $X$ by
\begin{equation}\label{eq:intrinsicmetric}
d_{\mathcal{E}}(x,y)=\sup\{u(x)-u(y)\, :\, u\in\mathcal{F}\cap C_0(X)\text{ and } d\Gamma(u,u)\le d\mu\},
\end{equation}
where the condition $d\Gamma(u,u)\le d\mu$ means that $\Gamma(u,u)$ is absolutely continuous with 
respect to $\mu$ with Radon-Nikodym derivative bounded by $1$. 

\begin{defn}
A strongly local regular Dirichlet form $(\mathcal{E},\mathcal{F})$ is called strictly local if $d_{\mathcal{E}}$ is a metric on $X$ and the topology induced by $d_{\mathcal{E}}$ coincides with the topology generated by the underlying metric of $(X,d)$.
\end{defn}

\subsection{$\Gamma$-convergence and Mosco convergence}
The two types of convergence of Dirichlet forms we deal with are (de Giorgi) $\Gamma$-convergence and Mosco convergence. To better suit our discussion, the definitions are presented in the context of $L^2$-spaces, although they are available in more general settings, see~\cite{dGF75,Mos94} and also~\cite{KS03} for generalizations of the original versions.
\begin{defn}\label{D:Gamma_conv}[$\Gamma$-convergence]
A sequence of forms $\{\mathcal{E}_n\}_{n\geq 1}$ is said to $\Gamma$-converge to $\mathcal{E}$ if 
\begin{enumerate}[wide=0em,label=(\alph*)]
    \item For any sequence $\{f_n\}_{n\geq 1}\subset L^2(X,\mu)$ that converges \emph{strongly} to $f\in L^2(X,\mu)$ in $L^2(X,\mu)$, 
    \[
    \liminf_{n\to\infty}\mathcal{E}_n(f_n,f_n)\geq \mathcal{E}(f,f).
    \]
    \item For any $f\in L^2(X,\mu)$ there exists a sequence $\{f_n\}_{n\geq 1}\subset L^2(X,\mu)$ that converges strongly to $f$ in $L^2(X,\mu)$ and
    \[
    \limsup_{n\to\infty}\mathcal{E}(f_n,f_n)\leq \mathcal{E}(f,f).
    \]
\end{enumerate}
\end{defn}
Note that in the previous definition, the domain of the $\Gamma$-limit $\mathcal E$ is then 
\[
\mathrm{dom} (\mathcal E)=\left\{ f \in L^2(X,\mu), \mathcal{E}(f,f) <+\infty \right\}
\]
Mosco convergence is slightly stronger and enjoys the property of being characterizable in terms of convergence of semigroups, resolvents and spectral families associated with the corresponding Dirichlet forms.
\begin{defn}\label{D:Mosco_conv}[Mosco convergence]
A sequence of forms $\{\mathcal{E}_n\}_{n\geq 1}$ is said to Mosco-converge to $\mathcal{E}$ if 
\begin{enumerate}[wide=0em,label=(\alph*)]
    \item For any sequence $\{f_n\}_{n\geq 1}\subset L^2(X,\mu)$ that converges \emph{weakly} to $f\in L^2(X,\mu)$ in $L^2(X,\mu)$, 
    \[
    \liminf_{n\to\infty}\mathcal{E}_n(f_n,f_n)\geq \mathcal{E}(f,f).
    \]
    \item For any $f\in L^2(X,\mu)$ there exists a sequence $\{f_n\}_{n\geq 1}\subset L^2(X,\mu)$ that converges strongly to $f$ in $L^2(X,\mu)$ and
    \[
    \limsup_{n\to\infty}\mathcal{E}(f_n,f_n)\leq \mathcal{E}(f,f).
    \]
\end{enumerate}
\end{defn}

\section{Strictly local Dirichlet forms on Cheeger spaces as Mosco limits of Korevaar-Schoen energies}\label{S:strictly_local}

This section starts off with a metric measure space as in the previous section equipped with a 2-Poincar\'e inequality for Lipschitz functions.

\subsection{Cheeger spaces}
A metric measure space $(X,d,\mu)$ that satisfies the volume doubling property and the 2-Poincar\'e inequality \eqref{E:p-Poincare} is often referred to as a \emph{Cheeger space} due to the very influential paper \cite{Che99}. We note that any Cheeger space is quasi-convex, see~\cite[Remark 8.3.3]{HKST15}.
The Lipschitz constant of a function $f\in{\rm Lip}(X)$ is defined as
\[
(\mathrm{Lip} f )(y):=\limsup_{r \to 0^+} \sup_{x \in X, d(x,y) \le r} \frac{|f(x)-f(y)|}{r}
\]
and the $2$-Poincar\'e inequality that will be assumed throughout this section is the following.

\begin{assump}[2-Poincar\'e inequality with Lipschitz constants]\label{A:2PI_Lip}
For any $f\in {\rm Lip}(X)$ and any ball $B(x,R)$ of radius $R>0$,
\begin{align}\label{E:p-Poincare}
\int_{B(x,R)} | f(y) -f_{B(x,R)}|^2 d\mu (y) \le C R^2 \int_{B(x,\lambda R)} (\mathrm{Lip} f )(y)^2 d\mu (y),
\end{align}
where
\[
f_{B(x,R)}:=\fint_{B(x,R)}  f(y) d\mu(y).
\]
The constants $C>0$ and $\lambda \ge 1$ in~\eqref{E:p-Poincare} are independent from $x$, $R$ and $f$.
\end{assump}

\begin{remark}
The 2-Poincar\'e inequality \eqref{E:p-Poincare} is equivalent to the 2-Poincar\'e inequality with upper gradients, see~\cite[Theorem 8.4.2]{HKST15}.
\end{remark}

On such a Cheeger space we are interested in the Korevaar-Schoen type energy functionals defined for any $f \in L^2(X,\mu)$ as
\begin{equation}\label{E:def_E2}
E(f,r):= \int_X\fint_{B(x,r)} \frac{|f(y)-f(x)|^2}{r^{2}} d\mu(y) d\mu(x),
\end{equation}
where in general we write
\[
\fint_{B(x,r)} f\, d\mu:=\frac{1}{\mu(B(x,r))} \int_{B(x,r)} f\, d\mu,
\]
and in the associated Korevaar-Schoen space 
\[
KS^{1,2}(X):=\left\{ f \in L^2(X,\mu), \, \limsup_{r \to 0^+} E(f,r) <+\infty \right\}.
\]

Functionals like~\eqref{E:def_E2} have been studied in~\cite{Bau22,Stu98a,Stu98b,KS05} in a similar set-up. 

\begin{remark}
In the present setting, the space $KS^{1,2}(X)$ equipped with the norm $\|f\|_{L^2} +\sup_{r >0} E(f,r)^{1/2}$ coincides with the Newtonian Sobolev space $N^{1,2}(X)$ with equivalent norm $\|f\|_{L^2} +\| g_f \|_{L^2}$, where $g_f$ is the minimal 2-weak upper gradient of $f$, see for instance \cite{MMS16} or \cite{Bau22}.
\end{remark}
\begin{remark}
From \cite{Bau22}, in our setting, there exists a constant $C>0$ such that for every $f \in KS^{1,2}(X)$,
\[
\sup_{r >0} E(f,r) \le C \liminf_{r \to 0} E(f,r).
\]
\end{remark}

\subsection{Main result}

Our aim is to prove, under the assumptions of the section, the following result.
\begin{theorem}\label{T:main}
There exists a  Dirichlet form $(\mathcal{E},\mathcal{F})$ on $L^2(X,\mu)$ such that:
\begin{enumerate}[wide=0em,label={\rm(\roman*)}]
\item $\mathcal{E}$ has domain $\mathcal{F}=KS^{1,2}(X)$;
\item $\mathcal{E}$ is a $\Gamma$-limit of $E(f,r_n)$, where $r_n$ is a positive sequence such that $r_n \to 0$;
\item When $X$ is compact, $\mathcal{E}$ is also the Mosco limit of $E(f,r_n)$;
\item $(\mathcal{E},KS^{1,2}(X))$ is regular with core $ {\rm Lip}(X) \cap C_c(X)$, strictly local, and its intrinsic distance $d_\mathcal{E}$ is bi-Lipschitz equivalent to  the original distance $d$;
\item  $\mathcal{E}$ satisfies the 2-Poincar\'e inequality
\[
\int_{B(x,r)} | f(y) -f_{B(x,r)}|^2 d\mu (y) \le C r^2 \int_{B(x,\lambda r)} d\Gamma (f,f),
\]
where $\Gamma (f,f)$ is the energy measure of $f$.
\end{enumerate}

Moreover, on $KS^{1,2}(X)$
\[
 \mathcal{E}(f,f)  \simeq \liminf_{r \to 0^+} E(f,r) \simeq \sup_{r > 0} E(f,r).
\]
\end{theorem}

\begin{remark}
The Cheeger energy functional $\mathbf{Ch}_2 (f)$ is defined by
\[
\mathbf{Ch}_2 (f)=\int_X g_f^2 d\mu,
\]
where $g_f$ is the minimal 2-weak upper gradient of $f$ in the metric measure space $(X,d,\mu)$. One also has
\[
\mathbf{Ch}_2 (f) = \inf_{f_n} \int_{X} (\mathrm{Lip} f_n )(y)^2 d\mu (y),
\]
where the infimum $\inf_{f_n}$ is taken over the sequences of locally Lipschitz functions $f_n$ such that $f_n \to f$ in $L^2_{loc}(X,\mu)$. From~\cite[ Theorem 5.1]{Bau22}, the Dirichlet form $\mathcal E$ in Theorem~\eqref{T:main} therefore satisfies 
\[
\mathcal{E}(f,f) \simeq \mathbf{Ch}_2 (f).
\]

Note that in general $\mathbf{Ch}_2 (f)$ is not a Dirichlet form since it might fail to satisfy the  parallelogram identity, see \cite{KSZ14}.
\end{remark}

\begin{remark}
On $(X,d,\mu)$, one can also construct a strictly local Dirichlet form
\[
\tilde{\mathcal{E}}(f,f)=\int_X \| D f \|^2 d\mu
\]
using a Cheeger differentiable structure, see~\cite[Section 2.3]{MMS16}. This Dirichlet form also satisfies
\[
\tilde{\mathcal{E}}(f,f) \simeq \mathbf{Ch}_2 (f).
\]
Thus, the Dirichlet form $\mathcal E$ in Theorem \eqref{T:main} is such that 
\[
\mathcal{E}(f,f)  \simeq \tilde{\mathcal{E}}(f,f).
\]
From Le Jan theorem~\cite[Proposition 1.5.5(b)]{LJ} we therefore get
\[
\frac{d\Gamma (f,f)}{d\mu} \simeq g_f^2 \simeq \| D f \|^2.
\]
\end{remark}

\begin{remark}
It follows from~\cite{Gor22}, \cite{Stu98a}, \cite{KSZ14}, see also \cite{HP21}, that if $(X,d,\mu)$ is a RCD$(0,N)$ space then the  Dirichlet form $\mathcal{E}$ in Theorem \ref{T:main}  is unique and actually given by the pointwise limit
\[
\mathcal{E} (f,f)=\lim_{r \to 0} E(f,r)= C_N \mathbf{Ch}_2 (f)=C_N \int_X \| Df\|^2 d\mu,
\]
where $C_N$ is a universal constant and $\mathbf{Ch}_2 (f)=\int_X g_f^2 d\mu$ is the Cheeger energy  defined as  before and $\int_X \| Df\|^2 d\mu$ is the Dirichlet form obtained from the Cheeger differential structure. In general, we do not know if  the form $\mathcal E$ that satisfies the properties of Theorem \ref{T:main} is unique or not.
\end{remark}

%

The proof of Theorem \ref{T:main} is divided in several intermediary results according to the subsequent sections. 

\subsection{Existence of a $\Gamma$-limit}\label{S:Mosco_convergence}

We first derive the existence of a $\Gamma$-limit from~\cite[Theorem 2.1]{KS05}, see also~\cite[Theorem 3.3]{Stu98a}. Under the additional assumption that the underlying space $(X,d)$ is compact, we will obtain in Theorem~\ref{T:Mosco_limit} the convergence in the Mosco sense.

\medskip

We begin with an observation that will allow us to apply the results in~\cite{KS05}.

\begin{lem}\label{L:Kumagai-Sturm_condition}
Let $\{\eps_n\}_{n\geq 0}$ with $\eps_n >0$ and $\lim_{n\to\infty}\eps_n=0$. There exists a constant $C>0$ independent of the sequence $\{\eps_n\}_{n\geq 0}$ such that
\[
  \sup_{r>0} E(f,r) \le  C \liminf_{n \to +\infty}  E(f_n,\eps_n)
\]
for all $f\in L^2(X,\mu)$ and all $\{f_n\}_{n\geq 1}\subset L^2(X,\mu)$ with $f_n\to f$ in $L^2(X,\mu)$.
\end{lem}

\begin{proof}
Note first, see for instance the proof of~\cite[Theorem 5.1]{Bau22}, that the $2$-Poincar\'e inequality~\eqref{E:p-Poincare} implies that for every $r>0$ and $f\in {\rm Lip}_{\rm loc}(X)$,
\begin{align}\label{sup kor so}
\int_X\fint_{B(x,r)}\frac{|f(x)-f(y)|^2}{r^2}d\mu(y)\,d\mu(x) \le C \int_{X} (\mathrm{Lip} f )(y)^2 d\mu (y).
\end{align}


Let now $f \in L^2(X,\mu)$ and $\{f_n\}_{n\geq 1}{\subset} L^2(X,\mu)$ such that $f_n \to f$ in $L^2(X,\mu)$. Fix $\eps>0$ and let $\{B_i^\eps=B(x_i,\eps)\}_i$ denote an $\eps$-covering of $X$ such that the family $\{B_i^{5\eps}\}_i$ has the bounded overlap property uniformly in $\eps$. Moreover, consider $\{\pip_i^\eps\}_i$ a $(C/\eps)$-Lipschitz partition of unity $\{\pip_i^\eps\}_i$ subordinated to this cover: that is, 
$0\le \pip_i^\eps\le 1$ on $X$, $\sum_i\pip_i^\eps=1$ on $X$, and $\pip_i^\eps=0$ in $X\setminus B_i^{2\eps}$. Setting
\begin{equation}\label{E:def_f_eps}
f_{n,\eps}:=\sum_i f_{n,B_i^\eps}\, \pip_i^\eps,
\end{equation}
where $f_{n,B_i^\eps}:=\vint_{B_i^\eps} f_nd\mu$, we note first that $f_{n,\eps}$ is locally Lipschitz. Indeed, for any fixed $j$ and any $x,y\in B_j^\eps$,
\begin{align*}
|f_{n,\eps}(x)-f_{n,\eps}(y)|&\le \sum_{i:B_i^{2\eps}\cap B_j^{2\eps}\ne\emptyset}|f_{n,B_i^\eps}-f_{n,B_j^\eps}||\pip_i^\eps(x)-\pip_i^\eps(y)|\\
 &\le \frac{C\, d(x,y)}{\eps} \sum_{i:B_i^{2\eps}\cap B_j^{2\eps}\ne\emptyset}
    \left(\vint_{B_i^\eps}\vint_{B(x,2\eps)}|f_n(y)-f_n(x)|^2\, d\mu(y)\, d\mu(x)\right)^{1/2}.
\end{align*}
Therefore, locally on $B_j^\eps$,
\begin{align*}
\mathrm{Lip} (f_{n,\eps})&\le \frac{C}{\eps}\sum_{i:B_i^{2\eps}\cap B_j^{2\eps}\ne\emptyset}
    \left(\vint_{B_i^\eps}\vint_{B(x,2\eps)}|f_n(y)-f_n(x)|^2\, d\mu(y)\, d\mu(x)\right)^{1/2}\\
    &\le C\left(\vint_{5B_j^\eps} \vint_{B(x,2\eps)}\frac{|f_n(y)-f_n(x)|^2}{\eps^2}\, d\mu(y)\, d\mu(x)\right)^{1/2}.
\end{align*}
Due to the bounded overlap property of the collection $5B_j^\eps$, we further get
\begin{align}\label{E:L2_Lip_est_01}
\int_X\mathrm{Lip} (f_{n,\eps})^2\, d\mu &\le \sum_j \int_{B_j^\eps}\mathrm{Lip} (f_{n,\eps})^2\, d\mu\notag\\
  &\le C\, \sum_j \int_{5B_j^\eps} \vint_{B(x,2\eps)}\frac{|f_n(y)-f_n(x)|^2}{\eps^2}\, d\mu(y)\, d\mu(x)\notag\\
  &\le C\, \int_X \vint_{B(x,2\eps)}\frac{|f_n(y)-f_n(x)|^2}{\eps^2}\, d\mu(y)\, d\mu(x).
 \end{align}
Thus, it follows from~\eqref{sup kor so} and~\eqref{E:L2_Lip_est_01} that for any $r>0$
\begin{equation}\label{eq:sup-gradient}
\frac{1}{r^2}\int_X\fint_{B(x,r)}|f_{n,\eps}(x)-f_{n,\eps}(y)|^2d\mu(y)\,d\mu(x) \le C\, \int_X \vint_{B(x,2\eps)}\frac{|f_n(y)-f_n(x)|^2}{\eps^2}\, d\mu(y)\, d\mu(x)
\end{equation}
and hence for every sequence $\eps_n \to 0^+$,
\begin{multline*}
 \liminf_{n \to +\infty} \frac{1}{r^2}\int_X\fint_{B(x,r)}|f_{n,\eps_n/2}(x)-f_{n,\eps_n/2}(y)|^2d\mu(y)\,d\mu(x)\\
 \le C\, \liminf_{n \to +\infty}  \int_X \vint_{B(x,\eps_n)}\frac{|f_n(y)-f_n(x)|^2}{\eps_n^2}\, d\mu(y)\, d\mu(x)
 =C\,\liminf_{n \to +\infty}E(f_n,\varepsilon_n).
\end{multline*}
The proof will therefore be complete once we prove that $f_{n,\eps_n/2}\to f$ in $L^2$ (recall that by assumption only $f_n\to f$ in $L^2$).
In a similar manner as before, we can show for any $\varepsilon>0$ and any $x\in B_j^\eps$ that
\begin{align*}
    |f_{n,\eps}(x)-f_n(x)|&=\Big| \sum_{i:B_i^{2\eps}\cap B_j^{2\eps}\ne\emptyset}\varphi_i^\eps (x)(f_{n,B_i^\eps}-f_n(x))\Big|\\
    &\leq \sum_{i:B_i^{2\eps}\cap B_j^{2\eps}\ne\emptyset}\Big|\fint_{B_i^\eps}(f_n(y)-f_n(x))\,d\mu(y)\Big|\\
    &\leq \sum_{i:B_i^{2\eps}\cap B_j^{2\eps}\ne\emptyset}\fint_{B_i^\eps}|f_n(y)-f_n(x)|\,d\mu(y)\\
    &\leq C\fint_{B(x,6\eps)}|f_n(x)-f_n(y)|d\mu(y)
\end{align*}
whence
\begin{align}\label{E:L2_Lip_est_02}
\int_X|f_{n,\eps}(x)-f_n(x)|^2\, d\mu(x) & 
  \le C  \int_X\left( \fint_{B(x,6\eps)}|f_n(x)-f_n(y)|d\mu(y) \right)^2\,d\mu(x) \\
  &\leq C  \int_X\fint_{B(x,6\eps)}|f_n(x)-f(x)|^2d\mu(y)\,d\mu(x)  \notag\\
  &+ C  \int_X\left( \fint_{B(x,6\eps)}|f(x)-f(y)|d\mu(y)\right)^2\,d\mu(x)  \notag\\
  &+ C  \int_X\fint_{B(x,6\eps)}|f(y)-f_n(y)|^2d\mu(y)\,d\mu(x)  \notag\\
 & \le C \| f -f_n\|^2_{L^2(X,\mu)} +C  \int_X\left( \fint_{B(x,6\eps)}|f(x)-f(y)|d\mu(y)\right)^2\,d\mu(x) .\notag
\end{align}
Notice now that from the Lebesgue differentiation theorem and from the fact that the maximal function of $f$ is in $L^2(X,\mu)$ we obtain by dominated convergence:
\[
\lim_{\eps \to 0} \int_X\left( \fint_{B(x,6\eps)}|f(x)-f(y)|d\mu(y)\right)^2\,d\mu(x)=0.
\]
Substituting $\eps$ by $\eps_n/2$ above, we obtain $f_{n,\eps_n/2} \to f$ in $L^2(X,\mu)$ as $n\to\infty$ and the proof is complete.
\end{proof}

\begin{theorem}\label{T:KS_as_Gamma}
There exists a strongly local and regular  Dirichlet form $(\mathcal{E},KS^{1,2} (X))$ on $L^2(X,\mu)$ with core $ {\rm Lip}(X) \cap C_c(X)$ such that for every $f \in KS^{1,2} (X)$
\begin{equation}\label{E:KS_comp_Gamma}
C_1  \sup_{r >0} E(f,r) \le \mathcal{E}(f,f) \le C_2 \liminf_{r\to 0^+} E(f,r).
\end{equation}
In addition, for a suitable positive sequence $\{r_n\}_{n\geq 1}$ converging to zero,
\begin{equation}\label{E:KS_as_Gamma}
    \mathcal{E}(f,f)=\Gamma{-}\!\!\lim_{n\to\infty}E(f,r_n).
\end{equation}
\end{theorem}

\begin{proof}
By definition, the bilinear form $E(f,r)$ falls into the framework of~\cite{KS05} with the kernel $k_r(x,y)=r^{-2}\mu(B(x,r))^{-1}\mathbf{1}_{B(x,r)}(y)$. In view of~\cite[Remark 2]{KS05}, see also~\cite[Theorem 3.3]{Stu98a}  and Lemma ~\ref{L:Kumagai-Sturm_condition}, the conditions of~\cite[Theorem 2.1]{KS05} are satisfied, which provide the existence of a strongly local Dirichlet form with domain $KS^{1,2} (X)$ fulfilling~\eqref{E:KS_comp_Gamma} and~\eqref{E:KS_as_Gamma}. Finally, regularity comes from the fact that $\rm{Lip}(X) \cap C_c(X)$ is dense in $KS^{1,2} (X)$ for the norm $\| f \|_{L^2(X,\mu)} + \sup_{r >0} E(f,r)^{1/2}$ and dense in $C_c(X)$ for the supremum norm.

\end{proof}

\subsection{Mosco convergence}
While in general $\Gamma$-convergence is weaker than Mosco convergence, it is known~\cite[Lemma 2.3.2]{Mos94} that both convergences are equivalent when the sequence of forms is \emph{asymptotically compact}. In the present setting that means 
that  if $\{\eps_n\}_{n\geq 0}$ is a sequence with $\eps_n >0$ and $\lim_{n\to\infty}\eps_n=0$, then any sequence $\{f_n\}_{n\geq 1}\subset L^2(X,\mu)$ with 
\begin{equation*}
\liminf_{n\to\infty} (E(f_n,\varepsilon_n)+\|f_n\|_{L^2(X,\mu)}^2)<\infty
\end{equation*}
has a subsequence that converges strongly in $L^2(X,\mu)$. At this point we will assume $X$ to be compact and use a result obtained in~\cite[Theorem 8.1]{HK00}, which in particular says that a sequence $\{f_n\}_{n\geq 1}$ that satisfies the $2$-Poincar\'e inequality~\eqref{E:p-Poincare} with upper gradients $\{g_n\}_{n\geq 1}$ on the right hand side and
\begin{equation}\label{E:Koskela-Hajlasz-cond}
\sup_{n\geq 1}\big(\|f_n\|_{L^1(X,\mu)}+\|g_n\|_{L^2(X,\mu)}\big)<\infty    
\end{equation}
contains an $L^2$-convergent subsequence.

\begin{lem}\label{L:asymp_cpt_cond}
Assume that $X$ is compact. Let $\{\eps_n\}_{n\geq 0}$ with $\eps_n >0$ and $\lim_{n\to\infty}\eps_n=0$.  Any sequence $\{f_n\}_{n\geq 1}\subset L^2(X,\mu)$ such that 
\begin{equation}\label{E:asymp_cpt_cond}
\liminf_{n\to\infty} ( E(f_n,\varepsilon_n)+ \|f_n\|_{L^2(X,\mu)}^2) <+\infty
\end{equation} 
has a subsequence that converges strongly in $L^2(X,\mu)$.
\end{lem}
\begin{proof}

From \eqref{E:asymp_cpt_cond} one can extract an increasing sequence  $n_k \to +\infty$ such that
\[
\sup_{k \ge 1} ( E(f_{n_k},\varepsilon_{n_k})+ \|f_{n_k}\|_{L^2(X,\mu)}^2) <+\infty
\]
Consider the sequence $\{f_{n_k,\varepsilon_{n_k}/6}\}_{n\geq 1}$, where $f_{n_k,\varepsilon_{n_k}/6}$ is defined as in~\eqref{E:def_f_eps}. By construction, $f_{n_k,\varepsilon_{n_k}/6}$ is locally Lipschitz and thus~\cite[Lemma 6.2.6]{HKST15} implies $g_{n_k,\varepsilon_{n_k}/6}={\rm Lip}f_{n_k,\varepsilon_{n_k}/6}$ is an upper gradient of $f_{n_k,\varepsilon_{n_k}/6}$. 
 Since $X$ is compact, the sequence $f_{n_k}$ is bounded in $L^1(X,\mu)$, this easily implies from ~\eqref{E:def_f_eps} that $f_{n_k,\varepsilon_{n_k}/6}$ is also  bounded in $L^1(X,\mu)$. Moreover,
  it follows from~\eqref{E:L2_Lip_est_01}  that
\begin{align*}
    \|g_{n_k,\varepsilon_{n_k}/6}\|_{L^2(X,\mu)}^2&=\|{\rm Lip}f_{n_k,\varepsilon_{n_k}/6}\|_{L^2(X,\mu)}^2=\int_X({\rm Lip}f_{n_k,\varepsilon_{n_k}/6}(y))^2d\mu(y)\\
    &\leq C\int_X\fint_{B(x,\varepsilon_{n_k})}\frac{|f_{n_k}(y)-f_{n_k}(x)|^2}{\varepsilon_{n_k}^2}d\mu(y)\,d\mu(x).
\end{align*}

\medskip

We deduce

\[
\sup_{k \geq 1}\big(\|f_{n_k,\varepsilon_{n_k}/6}\|_{L^1(X,\mu)}+\|g_{n_k,\varepsilon_{n_k}/6}\|_{L^2(X,\mu)}\big)<\infty 
\]

From ~\cite[Theorem 8.1]{HK00}, we can therefore extract a subsequence that will still denote $n_k$ such that $f_{n_k,\varepsilon_{n_k}/6}$ converges in $L^2(X,\mu)$. Let  us denote  by $f$ this limit.

We claim that  $\{f_{n_k}\}_{k\geq 1}$ converges to $f$ in $L^2(X,\mu)$. Indeed, it follows from~\eqref{E:L2_Lip_est_02} that
\begin{align}\label{E:subseq_conv_Gaussian}
    \|f-f_{n_k}\|_{L^2(X,\mu)}&\leq \|f-f_{n_k,\varepsilon_{n_k}{/6}}\|_{L^2(X,\mu)}+\|f_{n_k,\varepsilon_{n_k}/6}-f_{n_k}\|_{L^2(X,\mu)}\notag\\
    &\leq \|f-f_{n_k,\varepsilon_{n_k}/6}\|_{L^2(X,\mu)}+C\int_X\fint_{B(x,\varepsilon_{n_k})}|f_{n_k}(x)-f_{n_k}(y)|^2\mu(y)\,d\mu(x)\notag\\
    &=\|f-f_{n_k,\varepsilon_{n_k}{/6}}\|_{L^2(X,\mu)}+C\varepsilon_{n_k}^2E(f_{n_k},\varepsilon_{n_k}).
\end{align}
It follows that $\{f_{n_k}\}_{k\geq 1}$ converges to $f$ in $L^2(X,\mu)$ because $\sup_{k \ge 1} E(f_{n_k},\varepsilon_{n_k}) <+\infty$ and $\varepsilon_{n_k}$ goes to zero when $k \to +\infty$.

\end{proof}
%
%
\begin{cor}\label{C:KS_rel_cpt}
Let $X$ be compact. Then $KS^{1,2}(X)$ is relatively compact in $L^2(X,\mu)$.
\end{cor}
This last corollary is the key step to obtain the Mosco convergence of the Dirichlet forms from Theorem~\ref{T:KS_as_Gamma}.
\begin{theorem}\label{T:Mosco_limit}
If $X$ is compact, there exists a strongly local Dirichlet form $(\mathcal{E},KS^{1,2} (X))$ on $L^2(X,\mu)$ such that for every $f \in KS^{1,2} (X)$
\begin{equation}\label{E:KS_comp_Mosco}
C_1  \sup_{r >0} E(f,r) \le \mathcal{E}(f,f) \le C_2 \liminf_{r >0} E(f,r).
\end{equation}
This Dirichlet form is the Mosco limit of the sequence in~\eqref{E:KS_as_Gamma}.
\end{theorem}
\begin{proof}
The claim follows from Theorem~\ref{T:KS_as_Gamma}, Lemma~\ref{L:asymp_cpt_cond} and~\cite[Lemma 2.3.2]{Mos94}.
\end{proof}
 
%

\subsection{2-Poincar\'e inequality with respect  to the energy measures}
One of the principal characteristics of a Cheeger space is the 2-Poincar\'e from Assumption~\ref{A:2PI_Lip}. However, in the Dirichlet space literature including the seminal work of Sturm~\cite{Stu98a,Stu98b}, the right hand side of the assumed 2-Poincar\'e inequality involves the engergy measure of the function instead of the Lipschitz constant. In this section we prove this 2-Poincar\'e with respect to the energy measures.

\begin{theorem}\label{T:2PI_Gamma}
There exists $C>0$ and $\Lambda>1$ such that 
\begin{equation}\label{E:2PI_Gamma}
\int_{B(x,R)}|f(y)-f_{B(x,R)}|^2d\mu(z)\leq CR^2\int_{B(x,\Lambda R)}d\Gamma(f,f)
\end{equation}
for any $f\in KS^{1,2}(X)$, $x\in X$ and $R>0$.
\end{theorem}

First, we use the locality of the Dirichlet form $(\mathcal{E},KS^{1,2}(X))$.
\begin{prop}\label{P:limsup_vs_Gamma}
There exists a constant $C>0$ such that 
\begin{equation}\label{E:limsup_vs_Gamma}
\limsup_{r\to 0^+}\frac{1}{r^2}\int_{B(x,R)}\fint_{B(z,r)}|f(z)-f(y)|^2d\mu(y)\,d\mu(z)\leq C\int_{B(x,\lambda R)}d\Gamma(f,f)
\end{equation}
for any $f\in KS^{1,2}(X)\cap C_0(X)$, $x\in X$, $\lambda>1$ and $R>0$.
\end{prop}
\begin{proof}
By virtue of~\cite[Remark 1(f)]{KS05}, there exists $C>0$ such that
\begin{equation}\label{E:limsup_vs_Gamma_01}
\int_X\phi^2 \,d\Gamma(f,f)\geq C\limsup_{r\to 0^+}\frac{1}{r^2}\int_X\fint_{B(x,R)}|f(x)-f(y)|^2\phi(x) \phi (y) \,d\mu(y)\,d\mu(x)
\end{equation}
for any $x\in X$, $R>0$, $f\in KS^{1,2}(X)\cap C_0(X)$ and non-negative $\phi\in C_0(X)$. Choosing $\phi$ to be a cutoff function with $0 \le \phi \le 1$,  $\phi\equiv 1$ on $B(x,R)$ and $\supp\phi\subseteq B(x,\lambda R)$, $\lambda>1$ we obtain
\begin{align*}
    &\limsup_{r\to 0^+}\frac{1}{r^2}\int_{B(x,R)}\fint_{B(z,r)}|f(z)-f(y)|^2d\mu(y)\,d\mu(z)\\
    &\leq \limsup_{r\to 0^+}\frac{1}{r^2}\int_{B(x,R)}\fint_{B(x,r)}|f(z)-f(y)|^2\phi(z)\phi(y) d\mu(y)\,d\mu(z)\\
    &\leq \limsup_{r\to 0^+}\frac{1}{r^2}\int_{X}\fint_{B(x,r)}|f(z)-f(y)|^2 \phi (z)\phi(y) d\mu(y)\,d\mu(z)\\
    &\leq C\int_X\phi^2 \,d\Gamma(f,f)=C\int_{B(x,\lambda R)}\phi^2\,d\Gamma(f,f)\\
    &\leq C\int_{B(x,\lambda R)}d\Gamma(f,f).
\end{align*}
\end{proof}

We now proceed to the proof of the 2-Poincar\'e inequality~\eqref{E:2PI_Gamma}, first for  locally Lipschitz functions with compact support, and afterwards for any function in $KS^{1,2}(X)$.

\begin{proof}[Proof of Theorem~\ref{T:2PI_Gamma}]
\begin{enumerate}[wide=0em,label=\emph{Step} \arabic*:,itemsep=.5em]
\item Let $f\in {\rm Lip}_{\rm loc}(X)\cap C_0(X)$. As in the proof of Lemma~\ref{L:Kumagai-Sturm_condition}, choose a $\varepsilon$-covering of $X$ and define
\begin{equation}\label{E:def_f_eps_2}
f_{\eps}:=\sum_i f_{B_i^\eps}\, \pip_i^\eps,
\end{equation}
where $f_{B_i^\eps}=\vint_{B_i^\eps} fd\mu$. Analogously to that proof of Lemma~\ref{L:Kumagai-Sturm_condition}, $f_{\eps}$ is locally Lipschitz and 
\begin{equation}\label{E:limsup_vs_Lip_01}
    ({\rm Lip}f_{\varepsilon})^2\leq \frac{C}{\varepsilon^2}\fint_{5B_i^\varepsilon}\fint_{B(z,2\varepsilon)}|f(z)-f(y)|^2d\mu(y)\,d\mu(z)
\end{equation}
on each $B_i^\varepsilon$. 

\item  Then, for $R>0$ and $0 <\varepsilon <R$
\begin{align}
    \int_{B(x,R)}({\rm Lip}f_{\varepsilon})^2d\mu&\leq \sum_{i, B_i^\varepsilon \cap B(x,R) \neq \emptyset  } \int_{B_i^\varepsilon}({\rm Lip}f_{\varepsilon})^2(x)d\mu(x)\notag\\
    &\leq C \sum_{i, B_i^\varepsilon \cap B(x,R) \neq \emptyset  }\int_{B_i^\varepsilon}\frac{1}{\varepsilon^2}\fint_{5B_i^\varepsilon}\fint_{B(z,2\varepsilon)}|f(z)-f(y)|^2d\mu(y)\,d\mu(z)\,d\mu(x)\notag\\
    &\leq C\sum_{i, B_i^\varepsilon \cap B(x,R) \neq \emptyset  } \int_{5B_i^\varepsilon}\frac{1}{\varepsilon^2}\fint_{B(z,2\varepsilon)}|f(z)-f(y)|^2d\mu(y)\,d\mu(z)\notag\\
    &\leq C\int_{B(x,7R)}\frac{1}{\varepsilon^2}\fint_{B(z,2\varepsilon)}|f(z)-f(y)|^2d\mu(y)\,d\mu(z).\label{E:limsup_vs_Lip_02}
\end{align}

\item
By construction, see also the proof of Lemma~\ref{L:Kumagai-Sturm_condition}, we have $f_{\varepsilon/2}\to f$ in $L^2(X,\mu)$ as $\varepsilon\to 0^+$. Further, the triangle inequality implies
\begin{align}
    \int_{B(x,R)}|f(z)-f_{B(x,R)}|^2d\mu(z)&\leq 2\int_{B(x,R)}|f(z)-f_{\varepsilon/2}(z)|^2d\mu(z)\notag\\
    &+2\int_{B(x,R)}|f_{\varepsilon/2}(z)-(f_{n,\varepsilon/2})_{B(x,R)}|^2d\mu(z)\notag\\
    &+2\int_{B(x,R)}|(f_{\varepsilon/2})_{B(x,R)}-f_{B(x,R)}|^2d\mu(z).\label{E:limsup_vs_Lip_03}
\end{align}

\item The first term in~\eqref{E:limsup_vs_Lip_03} is bounded by $\|f-f_{\varepsilon/2}\|_{L^2(X,\mu)}^2$ and the third also using Cauchy-Schwarz inequality because
\begin{align*}
 \int_{B(x,R)}|(f_{\varepsilon/2})_{B(x,R)}-f_{B(x,R)}|^2d\mu(z) 
 &= \mu(B(x,R))\Big|\fint_{B(x,R)}(f_{\varepsilon/2}(y)-f(y))\,d\mu(y)\Big|^2\notag\\
 &\leq \|f-f_{\varepsilon/2}\|_{L^2(X,\mu)}^2.
\end{align*}

\item
For the second term in~\eqref{E:limsup_vs_Lip_03}, since $f_{\varepsilon}\in{\rm Lip}_{\rm loc}(X)\cap C_0(X)$, the 2-Poincar\'e inequality~\eqref{E:p-Poincare} and~\eqref{E:limsup_vs_Lip_02} imply
\begin{align*}
    \int_{B(x,R)}|f_{\varepsilon/2}(z)-(f_{\varepsilon/2})_{B(x,R)}|^2d\mu(z)
    &\leq CR^2\int_{B(x,\lambda R)}({\rm Lip}f_{\varepsilon/2})^2d\mu\notag\\
    &\leq CR^2\int_{B(x,7\lambda R)}\frac{1}{\varepsilon^2}\int_{B(x,\varepsilon)}|f(z)-f(y)|^2d\mu(y)\,d\mu(z).
\end{align*}

\item Combining the last two steps with~\eqref{E:limsup_vs_Lip_03} yields
\begin{multline*}
 \int_{B(x,R)}|f(z)-f_{B(x,R)}|^2d\mu(z)\\
 \leq 2 \|f-f_{\varepsilon/2}\|_{L^2(X,\mu)}^2+ CR^2\int_{B(x,7\lambda R)}\frac{1}{\varepsilon^2}\int_{B(x,\varepsilon)}|f(z)-f(y)|^2d\mu(y)\,d\mu(z)
\end{multline*}
and taking $\limsup_{\varepsilon\to 0^+}$ on both sides it follows from Proposition~\ref{P:limsup_vs_Gamma} that, for $f\in {\rm Lip}_{\rm loc}(X)\cap C_0(X)$,
\begin{equation}\label{E:limsup_vs_Lip_04}
     \int_{B(x,R)}|f(z)-f_{B(x,R)}|^2d\mu(z)\leq CR^2\int_{B(x,\lambda'R)}d\Gamma (f,f),
\end{equation}
for some $\lambda'>1$.
\item Let now $f\in KS^{1,2}(X)$. Recall that ${\rm Lip}_{\rm loc}(X)\cap C_0(X)$ is a core for the Dirichlet form $(\mathcal{E},KS^{1,2}(X))$, hence there is a sequence $\{f_n\}_{n\geq 0} \subset {\rm Lip}_{\rm loc}(X)\cap C_0(X)$ that converges to $f$ in the $\mathcal{E}_1$ norm. Applying the triangle inequality and~\eqref{E:limsup_vs_Lip_04} we obtain
\begin{align}
    \int_{B(x,R)}|f(z)-f_{B(x,R)}|^2d\mu(z)&\leq 2\int_{B(x,R)}|f(z)-f_n(z)|^2d\mu(z)\notag\\
    &+2\int_{B(x,R)}|f_n(z)-(f_n)_{B(x,R)}|^2d\mu(z)\notag\\
    &+2\int_{B(x,R)}|(f_n)_{B(x,R)}-f_{B(x,R)}|^2d\mu(z)\notag\\
    &\leq 4 \|f-f_n\|_{L^2(X,\mu)}^2
    + CR^2\int_{B(x,\lambda'R)}d\Gamma (f_n,f_n)\label{E:limsup_vs_Lip_05}
\end{align}
with $\lambda'>1$ possibly different than $\lambda$. 
\item Consider a cutoff function $\phi$ with $\phi\equiv 1$ in $B(x,\lambda'R)$ and $\supp\phi\subseteq B(x,\tilde{\lambda}R)$ for some $\tilde{\lambda}>\lambda'$. Then,
\begin{align*}
    \int_{B(x,\lambda'R)}d\Gamma (f_n,f_n)\leq \int_{X}\phi\, d\Gamma (f_n,f_n)
\end{align*}
converges to $\int_{X}\phi d\Gamma (f,f)$ as $n\to\infty$ because
\begin{equation*}
    \bigg|\Big(\int_X\phi\, d\Gamma (f_n,f_n)\Big)^{1/2}-\Big(\int_X\phi\, d\Gamma (f,f)\Big)^{1/2}\bigg|\leq \|\phi\|_\infty^{1/2}\mathcal{E}(f_n-f,f_n-f)^{1/2}
\end{equation*}
see e.g.~\cite[p.123]{FOT11}, and the right hand side vanishes as $n\to \infty$.
\item Taking $n\to \infty$ in~\eqref{E:limsup_vs_Lip_05} finally yields
\begin{equation*}
    \int_{B(x,R)}|f(z)-f_{B(x,R)}|^2d\mu(z)\leq CR^2\int_{X}\phi d\Gamma (f,f)\leq CR^2\int_{B(x,\tilde{\lambda}R)}d\Gamma (f,f)
\end{equation*}
as we wanted to prove.
\end{enumerate}
\end{proof}




\subsection{Energy measures and  strict locality of  the Dirichlet form}
The next ingredient in proving the strict locality of the Dirichlet form $(\mathcal{E},KS^{1,2}(X))$ consists in showing that the energy measure $\Gamma(f,f)$ associated with a locally Lipschitz function $f\in KS^{1,2}(X)$ is absolutely continuous with respect to the underlying measure $\mu$. This is the statement of Theorem~\ref{T:Gamma_vs_Lip}, whose proof is inspired by~\cite[Proposition, p.389]{Mos94}. 

\medskip

Notice that, at this stage, we do \emph{not} know yet that the topology induced by the intrinsic distance of the Dirichlet form is the same as the one induced by the original metric. When the latter is assumed, Theorem~\ref{T:Gamma_vs_Lip} is a consequence of Theorem~\ref{T:2PI_Gamma} and the volume doubling property of $\mu$, c.f.~\cite[Lemma 2.11]{BV2}. In the present case it is possible to obtain a similar bound for the Radon-Nikodym derivative. 

\begin{theorem}\label{T:Gamma_vs_Lip}
For any $f\in {\rm Lip}_{\rm loc}(X) \cap KS^{1,2}(X)$,
\begin{equation}\label{E:Gamma_vs_Lip}
    d\Gamma(f,f)\leq C({\rm Lip}f)^2d\mu.
\end{equation}
\end{theorem}
\begin{proof}
Let $f\in {\rm Lip}_{\rm loc}(X)\cap KS^{1,2}(X)$ be fixed. We will prove~\eqref{E:Gamma_vs_Lip} by showing that for any $\phi\in {\rm Lip}_{\rm loc}(X)\cap C_0(X)$
\begin{equation}\label{E:Gamma_vs_Lip_b}
    \int_X\phi^2 d\Gamma(f,f)\leq 2\int_X\phi^2({\rm Lip}f)^2d\mu.
\end{equation}
In view of~\cite[p.389]{Mos94} and applying Theorem~\ref{T:KS_as_Gamma}, 
\begin{align}
\int_X\phi^2 d\Gamma(f,f)&=\lim_{\lambda\to\infty}\frac{1}{\lambda^2}\Big(\mathcal{E}(\phi\cos(\lambda f),\phi\cos(\lambda f))+\mathcal{E}(\phi\sin(\lambda f),\phi\sin(\lambda f))\Big)\notag\\
&\leq \lim_{\lambda\to\infty}\frac{C}{\lambda^2}\liminf_{r\to 0^+}\Big(E(\phi\cos(\lambda f),r)+E(\phi\sin(\lambda f),r)\Big).\label{E:Mosco_trick_01}
\end{align}
Further, note that for any $x\in X$ and $y\in B(x,r)$, 
\begin{align*}
    &\big(\phi(x)\cos(\lambda f(x))-\phi(y)\cos(\lambda f(y))\big)^2\\
    &=\Big(\phi(x)\big(\cos (\lambda f(x))-\cos (\lambda f(y))\big)+\cos (\lambda f(y))\big(\phi(x)-\phi(y)\big)\Big)^2\\
    &\leq 2\phi(x)^2\big(\cos (\lambda f(x))-\cos (\lambda f(y))\big)^2+2\cos^2(\lambda f(y))\big(\phi(x)-\phi(y)\big)^2\\
    &\leq 2\lambda^2 \phi(x)^2\sup_{y \in B(x,r)} |f(x)-f(y)|^2+2\sup_{y \in B(x,r)} |\phi(x)-\phi(y)|^2
\end{align*}
and the same holds with $\sin$ instead of $\cos$. Combining this estimate with the definition of $E(\cdot,r)$ and Fatou lemma, it follows from~\eqref{E:Mosco_trick_01} that 
\begin{align*}
  \int_X\phi^2 d\Gamma(f,f)&\leq \lim_{\lambda\to\infty}\frac{C}{\lambda^2}\liminf_{r\to 0^+}\frac{1}{r^2}\int_X\phi^2(x)\fint_{B(x,r)} \lambda^2 \sup_{y \in B(x,r)} |f(x)-f(y)|^2d\mu(y)\,d\mu(x)\\
  &+\lim_{\lambda\to\infty}\frac{C}{\lambda^2}\liminf_{r\to 0^+}\frac{1}{r^2}\int_X\fint_{B(x,r)}\sup_{y \in B(x,r)} |\phi(x)-\phi(y)|^2d\mu(y)\,d\mu(x)\\
  &\leq C\int_X\phi^2({\rm Lip}f)^2d\mu+\lim_{\lambda\to\infty}\frac{C}{\lambda^2}\int_X({\rm Lip}\phi)^2d\mu\\
  &=C\int_X\phi^2({\rm Lip}f)^2d\mu
\end{align*}
as we wanted to prove.
\end{proof}

Strict locality will follow from the bi-Lipschitz equivalence of the distance $d$ of the underlying space $(X,d,\mu)$ and the distance $d_{\mathcal{E}}$ induced by the Dirichlet form $(\mathcal{E},KS^{1,2}(X))$ given by
\begin{equation*}
    d_{\mathcal{E}}(x,y)=\sup\{f(x)-f(y)\colon f\in KS^{1,2}(X)\cap C_0(X)\text{ with }d\Gamma(f,f)\leq d\mu\}.
\end{equation*}
The proof makes use of the $(1,2)$-Poincar\'e inequality that readily follows from the $(2,2)$-Poincar\'e inequality from Theorem~\ref{T:2PI_Gamma}.
\begin{theorem}\label{T:bilipschitz_dist}
The distance induced by the Dirichlet form $(\mathcal{E},KS^{1,2}(X))$ is bi-Lipschitz equivalent to the distance $d$ of the underlying space. In  particular, the Dirichlet form $(\mathcal{E},KS^{1,2}(X))$ is strictly local.
\end{theorem}

\begin{proof}
By applying the telescopic argument of~\cite[Lemma 5.15]{HK98}, the Poincar\'e inequality in Theorem \ref{T:2PI_Gamma} implies that for $f\in KS^{1,2}(X)$, and $\mu$-a.e. $x,y \in X$,
\[
|f(x)-f(y)| \le C d(x,y) (g_f(x) +g_f(y)),
\]
where $g_f$ is the maximal function defined as
\[
g_f(x)=\sup_{r>0} \frac{1}{\mu(B(x,r))^{1/2}} \left(\int_{B(x,r)} d\Gamma(f,f)\right)^{1/2}.
\]
We deduce that for $f\in KS^{1,2}(X)\cap C_0(X)$ with   $d\Gamma(f,f)\leq d\mu$ it holds that $|f(x)-f(y)| \le C d(x,y)$ and therefore 
\begin{align*}
d_{\mathcal{E}}(x,y)&=\sup\{f(x)-f(y)\colon f\in KS^{1,2}(X)\cap C_0(X)\text{ with }d\Gamma(f,f)\leq d\mu\} 
 \le C d(x,y).
\end{align*}
On the other hand, by virtue of Theorem~\ref{T:Gamma_vs_Lip}, for any $x,y\in X$ we have
\begin{align*}
    d_{\mathcal{E}}(x,y)&=\sup\{f(x)-f(y)\colon f\in KS^{1,2}(X)\cap C_0(X)\text{ with }d\Gamma(f,f)\leq d\mu\}\\
    &\ge \sup\{f(x)-f(y)\colon f\in {\rm Lip}_{\rm loc}(X)\cap C_0(X) \cap  KS^{1,2}(X), C{\rm Lip}f \le 1 \}\\
    &\ge \frac{1}{C} d(x,y).
\end{align*}
\end{proof}

\section{Strongly local Dirichlet forms on fractal-like spaces as Mosco limits of Korevaar-Schoen energies }\label{S:non-strictly}

The goal of this section is to study Mosco limits of the Korevaar-Schoen energies in fractal-like spaces for which the 2-Poincar\'e inequality \eqref{A:2PI_Lip} is not available. Throughout the section $(X,d,\mu)$ is a metric measure space that satisfies the assumptions of Section~\ref{MMSintro} and is in addition compact. We note that the compactness of $X$ and doubling implies that for every $x \in X$, $0\le r \le R=\rm {diam} (X)$,
\begin{equation}\label{E:lower_mass_bound}
\mu(B(x,r)) \ge c r^Q,
\end{equation}
where $0<Q<\infty$ is the exponent in~\eqref{eq:mass-bounds}.

\subsection{Fractal-like spaces}\label{SS:Fractal_sps}
Let $d_w>2$ and consider the Korevaar-Schoen space 
\[
KS^{d_w/2,2}(X)=\left\{ f \in L^2(X,\mu), \, \limsup_{r \to 0^+} E_{d_w/2,X}(f,r) <+\infty \right\},
\] 
where
\[
E_{d_w/2,U}(f,r):= \int_U\frac{1}{\mu(B(x,r))} \int_{B(x,r)} \frac{|f(y)-f(x)|^2}{r^{d_w}} d\mu(y) d\mu(x)
\]
for any Borel set $U\subseteq X$. To shorten the notation and when no confusion may occur, $E_{d_w/2,X}(f,r)$ will be simply denoted as $E_{d_w/2}(f,r)$. At this level of generality, $KS^{d_w/2,2}(X)$ might only contain a.e. constant functions, but this will not affect the results of the next subsection \ref{S:Rellich-Kondrachov}. Later, we will need  the  further  assumption \ref{A:CCC} to ensure that $KS^{d_w/2,2}(X)$ contains enough functions.

The main assumption we work with in this section is the following:
%

\begin{assump}[$2$-Poincar\'e inequality with Korevaar-Schoen energies]\label{A:2PI_KS_dw}
For any $f\in KS^{d_w/2,2}(X)$,
\begin{equation}\label{E:2PI_KS_dw}
\int_{B(x,R)} |f(y)-f_{B(x,R)}|^2 d\mu (y) \le C R^{d_w} \liminf_{r\to 0^+} E_{d_w/2,B(x,\lambda R)}(f,r).
\end{equation}
The constants $C>0$ and $\lambda \ge 1$ in~\eqref{E:2PI_KS_dw} are independent from $x$, $R$ and $f$.
\end{assump} 

\begin{remark}
From Theorem \ref{converse}, a large class of examples of metric measure spaces satisfying \ref{A:2PI_KS_dw} are given by nested fractals or some infinitely ramified fractals like the Sierpinski carpet.
\end{remark}

\subsection{Preliminary: A Rellich-Kondrachov type theorem in Korevaar-Schoen spaces}\label{S:Rellich-Kondrachov}
%

We first prove an analogue of the Rellich-Kondrachov theorem \cite[Theorem 8.1]{HK00} in the context of fractal-like spaces. We start with the following lemma.

\begin{lem}\label{L:4KS_cond_dw}
There exists $C>0$ such that for every $f \in KS^{d_w/2,2}(X)$
\begin{equation}\label{E:sup_vs_limif_dw}
\sup_{r >0}E_{d_w/2,X}(f,r) \le C \liminf_{r\to 0^+}E_{d_w/2,X}(f,r).
\end{equation}
\end{lem}

\begin{proof}
The proof is a minor modification of that in~\cite[Theorem 3.1]{MMS16}, see also \cite[Theorem 5.2]{Bau22}; the details are included here for completeness. 
Fix $r>0$ and consider a $r$-covering of $X$ that consists of balls $\{B(x_i,r)\}_{i}$ with the property that  $\{B(x_i,2\lambda r)\}_{i}$ have the bounded overlap property, see e.g.~\cite[pp.102-103]{HKST15}. In particular, there exists $C>0$ independent of $r$ such that
\[
\sum_{i}\mathbf{1}_{B(x_i,2\lambda r)}(x)<C
\]
for all $x\in X$. In addition, for any $x\in B(x_i,r)$ and $y\in B(x,r)$ the doubling property implies
\begin{align*}
&\mu(B(x_i,r))\leq \mu(B(y,4 r))\leq C \mu(B(y,r))\\
&\mu(B(x_i,r))\leq\mu(B(x,2r))\leq C \mu(B(x,r)).   
\end{align*}
For any $f \in KS^{d_w/2,2}(X)$ it thus holds that
\begin{align}
 & \frac{1}{r^{d_w}}\int_X\int_{B(x,r)}\frac{|f(x)-f(y)|^2}{\mu(B(x,r))}d\mu(y)\,d\mu(x) \notag\\
\leq &\frac{1}{r^{d_w}}\sum_{i}\int_{B(x_i,r)}\int_{B(x,r)}\frac{|f(x)-f(y)|^2}{\mu(B(x,r))}d\mu(y)\,d\mu(x)\notag\\
\leq &\frac{2}{r^{d_w}}\sum_{i}\int_{B(x_i,r)}\int_{B(x,r)}\frac{|f(x)-f_{B(x_i,r)}|^2}{\mu(B(x,r))} + \frac{|f(y)-f_{B(x_i,r)}|^2}{\mu(B(x,r))}d\mu(y)\,d\mu(x)\label{E:4KS_cond_dw_help}
\end{align}
Applying the $2$-Poincar\'e inequality~\eqref{E:2PI_KS_dw} to the first term yields
\begin{align*}
 &\sum_{i}\int_{B(x_i,r)}\fint_{B(x,r)}|f(x)-f_{B(x_i,r)}|^2 d\mu(y) \,d\mu(x) \\
 =&\sum_{i}\int_{B(x_i,r)} |f(x)-f_{B(x_i,r)}|^2 d\mu(x) \\
 \le &C r^{d_w} \sum_{i} \liminf_{\varepsilon \to 0} E_{d_w/2,B(x_i,\lambda r)}(f,\varepsilon) \le C r^{d_w} \liminf_{\varepsilon \to 0} E_{d_w/2,X}(f,\varepsilon).
\end{align*}
To bound second term, using first Fubini theorem and afterwards the doubling property we get
\begin{align*}
 & \sum_{i}\int_{B(x_i,r)}\fint_{B(x,r)} |f(y)-f_{B(x_i,r)}|^2d\mu(y)\,d\mu(x) \\
 \le & \sum_{i}\int_{B(x_i,2r)}\int_{B(y,r)} \frac{|f(y)-f_{B(x_i,r)}|^2}{\mu(B(x,r))}d\mu(x)\,d\mu(y) \\
 \le &C \sum_{i}\int_{B(x_i,2r)} |f(y)-f_{B(x_i,r)}|^2 d\mu(y). 
\end{align*}
Further, the $2$-Poincar\'e inequality implies
\begin{align*}
 & \int_{B(x_i,2r)} |f(y)-f_{B(x_i,r)}|^2 d\mu(y) \\
 \le & 2 \left(  \int_{B(x_i,2r)} |f(y)-f_{B(x_i,2r)}|^2 d\mu(y)+ \mu( B(x_i,2r))  |f_{B(x_i,2r)}-f_{B(x_i,r)}|^2 \right) \\
 \le & C  \left(  r^{d_w} \liminf_{\varepsilon \to 0} E_{d_w/2,B(x_i,2\lambda r)}(f,\varepsilon) + \mu( B(x_i,2r))  |f_{B(x_i,2r)}-f_{B(x_i,r)}|^2 \right).
\end{align*}
Finally, from H\"older's inequality and the $2$-Poincar\'e inequality again it follows that
\begin{align*}
 & \mu( B(x_i,2r))  |f_{B(x_i,2r)}-f_{B(x_i,r)}|^2 \\
 \le & C \int_{B(x_i,r)} | f(y) -f_{B(x_i,2r)}|^2 d\mu(y) \\
  \le & C \int_{B(x_i,2r)} | f(y) -f_{B(x_i,2r)}|^2 d\mu(y) \\
  \le &C r^{d_w}  \liminf_{\varepsilon \to 0 } E_{d_w/2,B(x_i,2\lambda r)}(f,\varepsilon).
\end{align*}
Combining all previous estimates with~\eqref{E:4KS_cond_dw_help} we obtain, for every $r>0$,
\begin{align*}\label{sup kor so}
\frac{1}{r^{d_w}}\int_X\int_{B(x,r)}\frac{|f(x)-f(y)|^p}{\mu(B(x,r))}d\mu(y)\,d\mu(x) \le C \liminf_{\varepsilon \to 0 } E_{d_w/2,X}(f,\varepsilon)
\end{align*}
as we wanted to prove.
\end{proof}

We now record a maximal theorem that is implicitly proved in~\cite[Section 5.2]{BC22} and include the proof for completeness. For any $f\in KS^{d_w/2,2}(X)$ and $R>0$ define the maximal function
\begin{equation}\label{E:def_max_fct_E}
    M_Rf(x):=\sup_{0<\rho<R}\bigg(\frac{1}{\mu(B(x,\rho))}\liminf_{r\to 0^+}E_{d_w/2,B(x,\rho)}(f,r)\bigg)^{1/2}.
\end{equation}
\begin{theorem}\label{T:loc_maximal_thm}
For any $R>0$ and $f\in KS^{d_w/2,2}(X)$ the maximal function $M_Rf$ is weak-$L^2$ bounded, that is for every $\lambda >0$
\[
\mu(\{x\in X\colon |M_Rf(x)|>\lambda\})\leq C\lambda^{-2}\liminf_{r\to 0^+}E_{d_w/2,B(x,R)}(f,r)
\]
for some $C>0$ independent of $R$ and $f$.
\end{theorem}
\begin{proof}
    For each $\lambda > 0$, let $E_\lambda:=\{x\in X\colon |M_Rf(x)|>\lambda\}$. By definition of $M_Rf$, for any $x\in E_\lambda$ there exists $0<\rho_x<R$ such that
    \[
   \lambda^{-2}\liminf_{r\to 0^+}E_{d_w/2,B(x,\rho_x)}(f,r)>\mu(B(x,\rho_x)).
    \]
    The family $\{B(x,\rho_x)\}_{x\in E_\lambda}$ is a covering of $E_\lambda$ by balls of diameter at most $2R$. By virtue of Vitali's covering theorem, or see e.g.~\cite[Theorem 14.12]{HK00}, one can extract a pairwise disjoint subcollection $\{B(x_i,\rho_i)\}_{i\geq 0}$ with the property that $E_\lambda\subset \bigcup_{i\geq 0}B(x_i,5\rho_i)$. Applying volume doubling finally yields
    \begin{align*}
        \mu(E_\lambda)&\leq\sum_{i\geq 0}\mu(B(x_i,5\rho_i))\leq C\sum_{i\geq 0}\mu(B(x_i,\rho_i))\\
        &\leq C\lambda^{-2}\sum_{i\geq 0}\liminf_{r\to 0^+}E_{d_w/2,B(x_i,\rho_i)}(f,r)\leq C\lambda^{-2} \liminf_{r\to 0^+}E_{d_w/2,X}(f,r)
    \end{align*}
    as we wanted to prove.
\end{proof}
One of the ingredients in the proof of the Rellich-Kondrachov theorem is a suitable Sobolev embedding from which depends on the relationship between $d_w$ and $Q$. We start with the most restrictive case $Q>d_w$.

\begin{theorem}\label{T:RK_Qgeqdw}
Assume $Q>d_w$. Any sequence $\{f_n \}_{n\geq 1}\subset KS^{d_w/2,2}(X)$ such that 
\begin{equation}\label{E:RK_cond_Qgeqdw}
\sup_{n \ge 1} ( \| f_n \|_{L^2(X,\mu)}+ \liminf_{r\to 0^+}E_{d_w/2,X}(f_n,r))<+\infty
\end{equation}
contains a subsequence that converges in $L^\alpha(X,\mu)$ for any $1 \le \alpha < \frac{2Q}{Q-d_w}$.
\end{theorem}

\begin{remark}\label{R:RK_Qgeqdw}
Note that $\frac{2Q}{Q-d_w}>2$ since $d_w>0$. Convergence in $L^2$ will be of relevance later to prove Mosco convergence, c.f. Corollary~\ref{C:Mosco_limit_dw}.
\end{remark}
\begin{proof}
Conditions~\eqref{E:sup_vs_limif_dw} and~\eqref{E:lower_mass_bound} allow to apply~\cite[Theorem 4.3]{Bau22} and obtain the Sobolev embedding
\begin{equation}\label{E:Sobolev_embedding_dw}
    \|f_n\|_{L^q(X,\mu)}\leq C\Big(\|f_n\|_{L^2(X,\mu)}+ \liminf_{r\to 0^+}E_{d_w/2,X}(f_n,r)^{1/2}\Big)
\end{equation}
for $q=\frac{2Q}{Q-d_w}$. In view of assumption~\eqref{E:RK_cond_Qgeqdw}, the sequence $\{f_n \}_{n\geq 1}$ is bounded in $L^q(X,\mu)$ and has therefore a subsequence $\{f_{n_k}\}_{k\geq 1}$ that weakly converges in $L^q$ to some $f\in L^q(X,\mu)$. To prove that this subsequence also strongly converges in $L^\alpha(X,\mu)$ for any $1 \le \alpha < \frac{2Q}{Q-d_w}$ we follow~\cite[Theorem 8.1]{HK00} and show that $\{f_{n_k}\}_{k\geq 1}$ converges to $f$ in measure, c.f.~\cite[Lemma 8.2]{HK00}. To do so requires the maximal theorem from Theorem~\ref{T:loc_maximal_thm}; we include the details for completeness.\\
Fix $\varepsilon>0$. For any $k\geq 1$ and $\rho>0$,
\begin{align*}
    \mu(\{x\in X\colon|f(x)-f_k(x)|>\varepsilon\})&\leq \mu(\{x\in X\colon|f(x)-f_{B(x,\rho)}|>\varepsilon/3\})\\
    &+\mu(\{x\in X\colon|f_{B(x,\rho)}-f_{k,B(x,\rho)}|>\varepsilon/3\})\\
    &+\mu(\{x\in X\colon|f_{k,B(x,\rho)}-f_k(x)|>\varepsilon/3\})\\
    &=:A_{\rho,\varepsilon}+B_{\rho,k,\varepsilon}+C_{\rho,k,\varepsilon},
\end{align*}
where $f_{k,B(x,\rho)}=\fint_{B(x,\rho)}f_kd\mu$. By virtue of Lebesgue differentiation theorem, see e.g.~\cite[Theorem 14.15]{HK00}, $A_{\rho,\varepsilon}\to 0$ as $\rho\to 0$. Moreover, $B_{\rho,k,\varepsilon}\to 0$ as $k\to\infty$ due to the fact that $\{f_{n_k}\}_{k\geq 1}$ converges weakly in $L^q(X,\mu)$. To analyze $C_{\rho,k,\varepsilon}$, a telescopic argument, volume doubling and the 2-Poincar\'e inequality~\eqref{E:2PI_KS_dw} yield
\begin{align*}
    |f_{k,B(x,\rho)}-f_k(x)|&\leq \sum_{m=0}^\infty|f_{k,B(x,2^{-m}\rho)}-f_{k,B(x,2^{-m-1}\rho)}|\\
    &\leq \sum_{m=0}^\infty\fint_{B(x,2^{-m-1}\rho)}|f_k-f_{k,B(x,2^{-m}\rho)}|d\mu\\
    &\leq \sum_{m=0}^\infty\frac{1}{\mu(B(x,2^{-m-1}\rho))}\int_{B(x,2^{-m}\rho)}|f_k-f_{k,B(x,2^{-m}\rho)}|d\mu\\
    &\leq C\sum_{m=0}^\infty(2^{-m}\rho)^{\frac{d_w}{2}}\Big(\frac{1}{\mu(B(x,2^{-m-1}\rho))}\liminf_{r\to 0^+}E_{d_w/2,B(x,\lambda2^{-m}\rho)}(f_k,r)\Big)^{1/2}\\
    &\leq C\rho^{\frac{d_w}{2}}\sum_{m=0}^\infty 2^{-\frac{md_w}{2}}\Big(\frac{1}{\mu(B(x,\lambda2^{-m}\rho))}\liminf_{r\to 0^+}E_{d_w/2,B(x,\lambda2^{-m}\rho)}(f_k,r)\Big)^{1/2}\\
    &\leq C\rho^{\frac{d_w}{2}}(M_{\lambda\rho}f_k(x))^{1/2},
\end{align*}
where
\[
M_{\lambda\rho}f_k(x):=\sup_{0<R<\lambda\rho}\frac{1}{\mu(B(x,R))}\liminf_{r\to 0^+}E_{d_w/2,B(x,R)}(f_k,r).
\]
Thus, by virtue of Theorem~\ref{T:loc_maximal_thm},
\begin{align*}
   C_{\rho,k,\varepsilon}&\leq \mu(\{x\in X\colon |M_{\lambda\rho}f_k(x)|>C\varepsilon^2\rho^{-d_w} \})\\
   &\leq C\frac{\rho^{d_w}}{\varepsilon^2}\liminf_{r\to 0^+}E_{d_w/2,X}(f_k,r)
\end{align*}
which tends to zero as $\rho\to 0^+$ uniformly on $k$ because of~\eqref{E:RK_cond_Qgeqdw}.
\end{proof}

When $0<Q<d_w$ or $d_w=Q$, the available embeddings are stronger than~\eqref{E:Sobolev_embedding_dw} and the corresponding result can be derived in a similar fashion with less restrictions.
\begin{theorem}\label{T:RK_Qledw}
Assume $0<Q\leq d_w$. Any sequence $\{f_n \}_{n\geq 1}\subset KS^{d_w/2,2}(X)$ such that 
\begin{equation}\label{E:RK_cond_Qledw}
\sup_{n \ge 1} (\| f_n \|_{L^2(X,\mu)}+ \liminf_{r\to 0^+}E_{d_w/2,X}(f_n,r))<+\infty
\end{equation}
contains a subsequence that converges in $L^\alpha(X,\mu)$ for any $1 \le \alpha < \infty$.
\end{theorem}
\begin{proof}
For $0<Q<d_w$ it holds that
\[
\|f\|_{L^\infty(X,\mu)}\leq C\big(\|f\|_{L^2(X,\mu)}+ \liminf_{r\to 0^+}E_{d_w/2,X}(f,r)^{1/2}\big)^{\theta}\|f\|_{L^2(X,\mu)}^{1-\theta}
\]
for $\theta=\frac{Q}{d_w}$ and any $f\in KS^{d_w/2,2}(X)$, c.f.~\cite[Theorem 4.3]{Bau22}. Since $X$ is compact, the latter and~\eqref{E:RK_cond_Qledw} imply that $\{f_n \}_{n\geq 1}$ is bounded. The same proof of Theorem~\ref{T:RK_Qgeqdw} after~\eqref{E:Sobolev_embedding_dw} applies and the result follows. In the case $d_w=Q$, the Trudinger-Moser inequality holds, c.f.~\cite[Theorem 4.3]{Bau22} and in particular implies that $\{f_n \}_{n\geq 1}$ is bounded in any $L^q(X,\mu)$. 
\end{proof}

As a conclusion, and in any case we have the following result:

\begin{cor}\label{C:Rellich-Kondrachev-L2}
Any sequence $\{f_n \}_{n\geq 1}\subset KS^{d_w/2,2}(X)$ such that 
\begin{equation*}
\sup_{n \ge 1} (\| f_n \|_{L^2(X,\mu)} +\liminf_{r\to 0^+}E_{d_w/2,X}(f_n,r))<\infty
\end{equation*}
contains a subsequence that converges in $L^2(X,\mu)$.
\end{cor}

\subsection{Main result}\label{SS:non-strictly}
We continue in the setting of section~\ref{SS:Fractal_sps} with the additional controlled cutoff condition that also appeared in \cite{Bau22}.
\begin{assump}[Controlled cutoff condition]\label{A:CCC}
For every $\eps >0$ there exists a covering $\{B_i^\eps=B(x_i,\eps)\}_i$ of $X$, so that $\{B_i^{5\eps}\}_i$ has the bounded overlap property (uniformly in $\eps$) and an associated family of functions $\pip_i^\eps$ satisfying
\begin{itemize}
\item $\pip_i^\eps \in KS^{d_w/2,2}(X)\cap C(X)$;
\item $0\le \pip_i^\eps\le 1$ on $X$;
\item $\sum_i\pip_i^\eps=1$ on $X$; 
\item $\pip_i^\eps=0$ in $X\setminus B_i^{2\eps}$;
\item $\limsup\limits_{r \to 0^+} E_{d_w/2,X}(\pip_i^\eps,r) \le C\varepsilon^{-d_w} \mu(B_i^{\eps})$.
\end{itemize}
\end{assump}

\begin{remark}
Even though the assumption \ref{A:CCC} might seem difficult to check at first, it is in essence only a capacity estimate requirement for balls,  and in practice the covering $B_i^\eps$ and the associated partition of unity $\pip_i^\eps$ are obtained using standard covering arguments in doubling metric measure spaces. Indeed, assume that for every ball $B$ with radius $\varepsilon$ one can find a non negative $\phi \in C(X) \cap KS^{d_w/2,2}(X)$  supported inside of $B$  with  $\phi=1$ on $ B/2$ such that
\[
 \limsup_{r\to 0^+} E_{d_w/2,X}(\phi,r)\le C \frac{\mu(B)}{\varepsilon^{d_w}}.
\]
Then assumption \ref{A:CCC} is easily proved to be satisfied using covering by balls satisfying the bounded overlap property as in Section 4.1 in \cite{HKST15}. In  particular, Proposition \ref{P:CSdw_CCC} yields many situations where the assumption \ref{A:CCC} is satisfied.
\end{remark}
%
%
%

Under these assumptions on the underlying space and its associated Korevaar-Schoen energy, this section shows the existence of a naturally associated Dirichlet form.
\begin{theorem}\label{T:main_dw}
There exists a Dirichlet form $(\mathcal{E},\mathcal{F})$ on $L^2(X,\mu)$ such that
\begin{enumerate}[wide=0em,label={\rm(\roman*)}]
\item $\mathcal{E}$ has domain $\mathcal{F}=KS^{d_w/2,2}(X)$;
\item$\mathcal{E}$ is a Mosco limit of $E_{d_w/2,X}(f,r_n)$ where $r_n$ is a positive sequence such that $r_n \to 0$;
\item $(\mathcal{E},KS^{1,2}(X))$ is strongly local and regular with core $KS^{1,2}(X)\cap C(X)$;
\item  $\mathcal{E}$ satisfies the 2-Poincar\'e inequality for the energy measures
\[
\int_{B(x,R)} | f(y) -f_{B(x,R)}|^2 d\mu (y) \le C R^{d_w} \int_{B(x,\lambda R)} d\Gamma (f,f).
\]
\end{enumerate}
\end{theorem}

\subsection{Existence of the Dirichlet form}
We start by proving the existence of the Dirichlet form $(\mathcal{E},KS^{d_w/2,2}(X))$ in $L^2(X,\mu)$ as a $\Gamma$-limit of forms in Proposition~\ref{P:KS_as_Gamma_dw}, which in the next section will be upgraded to a Mosco limit. The following preparatory lemma reproduces Lemma~\ref{L:Kumagai-Sturm_condition} in the strictly local case.

\begin{lem}\label{L:Kumagai-Sturm_condition_dw}
Let $\{\eps_n\}_{n\geq 0}$ with $\lim_{n\to\infty}\eps_n=0$. There exists a constant $C>0$ such that
\[
\liminf_{n \to +\infty}  E_{d_w/2,X}(f_n,\eps_n) \geq C\sup_{r>0} E_{d_w/2,X}(f,r) 
\]
for all $f\in L^2(X,\mu)$ and all $\{f_n\}_{n\geq 1}\subset L^2(X,\mu)$ such that $f_n\to f$ in $L^2(X,\mu)$. 
\end{lem}

\begin{proof}
Let $\{f_n\}\subset L^2(X,\mu)$ so that $f_n\to f$ in $L^2(X,\mu)$. Fix $\varepsilon>0$ and consider the $\varepsilon$-covering $\{B_i^{\varepsilon}\}_{i\geq 1}$ and the partition of unity $\{\varphi_i^\varepsilon\}_{i}$ from the controlled cutoff condition. Further, set
\begin{equation}\label{E:def_f_eps_dw}
f_{n,\varepsilon}:=\sum_if_{n,B_i^\varepsilon}\varphi_i^\varepsilon.
\end{equation}
Since $\varphi_i^\varepsilon\in KS^{d_w/2,2}(X)$ and $X$ compact, the linearity of $KS^{d_w/2,2}(X)$ implies that $f_{n,\varepsilon}\in KS^{d_w/2,2}(X)$. 
Note that once we prove
\begin{equation}\label{E:KS_cond_dw_help}
\frac{1}{r^{d_w}}\int_X\fint_{B(x,r)}|f_{n,\eps}(x)-f_{n,\eps}(y)|^2d\mu(y)\,d\mu(x) \le \frac{C}{\varepsilon^{d_w}} \int_X \fint_{B(x,6\eps)}|f_n(y)-f_n(x)|^2d\mu(y)\, d\mu(x)
\end{equation}
for any $r>0$, the claim will follow verbatim  the proof of Lemma~\ref{L:Kumagai-Sturm_condition} after~\eqref{eq:sup-gradient}.
To that end, we first show that
\begin{equation}\label{E:KS_cond_dw_help_a}
    \frac{1}{r^{d_w}}\int_X\fint_{B(x,r)}|f_{n,\eps}(x)-f_{n,\eps}(y)|^2d\mu(y)\,d\mu(x) \le C\liminf_{r\to 0^+}E_{d_w/2,X}(f_{n,\varepsilon},r)
\end{equation}
and secondly that
\begin{equation}\label{E:KS_cond_dw_help_b}
    \limsup_{r\to 0^+}E_{d_w/2,X}(f_{n,\varepsilon},r)\le \frac{C}{\varepsilon^{d_w}} \int_X \fint_{B(x,6\eps)}|f_n(y)-f_n(x)|^2d\mu(y)\, d\mu(x).
\end{equation}

Notice that~\eqref{E:KS_cond_dw_help_a} follows from the previous Lemma~\ref{L:4KS_cond_dw} applied to $f_{n,\varepsilon}$. To prove~\eqref{E:KS_cond_dw_help_b}, the finite overlap property implies for $x \in B_j^\varepsilon$, $y \in B_j^{2\varepsilon}$
\[
|f_{n,\varepsilon}(x)-f_{n,\varepsilon}(y)|^2\leq C\sum_{i\colon B_i^{2\varepsilon}\cap B_j^{2\varepsilon}\neq \emptyset}|f_{n,B_i^\varepsilon}-f_{n,B_j^\varepsilon}|^2|\varphi_i^\varepsilon(x)-\varphi_i^\varepsilon(y)|^2.
\]
Thus, for any $0<r <\varepsilon$
\begin{align*}
    &\frac{1}{r^{d_w}}\int_X\fint_{B(x,r)}|f_{n,\varepsilon}(x)-f_{n,\varepsilon}(y)|^2d\mu(y)\,d\mu(x)\\
    &\leq \frac{1}{r^{d_w}}\sum_{j}\int_{B_j^\varepsilon}\fint_{B(x,r)}|f_{n,\varepsilon}(x)-f_{n,\varepsilon}(y)|^2d\mu(y)\,d\mu(x)\\
    &\leq C\sum_j\sum_{i\colon B_i^{2\varepsilon}\cap B_j^{2\varepsilon}\neq \emptyset}|f_{n,B_i^\varepsilon}-f_{n,B_j^\varepsilon}|^2\frac{1}{r^{d_w}}\int_{B_j^\varepsilon}\fint_{B(x,r)}|\varphi_i^\varepsilon(x)-\varphi_i^\varepsilon(y)|^2d\mu(y)\,d\mu(x).
\end{align*}
Taking $\limsup_{r\to 0^+}$ on both sides of the inequality, 
the properties of the cutoff condition yield
\begin{align}
    \limsup_{r\to 0^+}E_{d_w/2,X}(f_{n,\varepsilon},r)&\leq C\sum_j\sum_{i\colon B_i^{2\varepsilon}\cap B_j^{2\varepsilon}\neq \emptyset}|f_{n,B_i^\varepsilon}-f_{n,B_j^\varepsilon}|^2\limsup_{r\to 0^+}E_{d_w/2,B_j^{\varepsilon}}(\varphi_i^\varepsilon,r)\notag\\
    &\leq \frac{C}{\varepsilon^{d_w}}\sum_j\sum_{i\colon B_i^{2\varepsilon}\cap B_j^{2\varepsilon}\neq \emptyset}|f_{n,B_i^\varepsilon}-f_{n,B_j^\varepsilon}|^2\mu(B_i^\varepsilon).\label{E:KS_cond_help_b1}
\end{align}
Further, Cauchy-Schwarz inequality and volume  doubling property imply
\begin{align*}
    |f_{n,B_i^\varepsilon}-f_{n,B_j^\varepsilon}|^2&=\Big(\fint_{B_j^\varepsilon}\big(f_{n,B_i^\varepsilon}-f(x)\big)d\mu(x)\Big)^2\\
    &\leq \fint_{B_j^\varepsilon}|f_{n,B_i^\varepsilon}-f_n(x)|^2d\mu(x)\\
    &\leq \fint_{B_j^\varepsilon}\fint_{B_i^\varepsilon}|f_n(y)-f_n(x)|^2d\mu(y)\,d\mu(x)\\
    &\leq C\fint_{B_j^\varepsilon}\fint_{B(x,6\varepsilon)}|f_n(y)-f_n(x)|^2d\mu(y)\,d\mu(x).
\end{align*}
Thus, it now follows from~\eqref{E:KS_cond_help_b1} and the bounded overlap property that
\begin{align*}
\limsup_{r\to 0^+}E_{d_w/2,X}(f_{n,\varepsilon},r)&\leq \frac{C}{\varepsilon^{d_w}}\sum_j \sum_{i\colon B_i^{2\varepsilon}\cap B_j^{2\varepsilon}\neq \emptyset}\mu(B_i^\varepsilon)\fint_{B_j^\varepsilon}\fint_{B(x,6\varepsilon)}|f_n(y)-f_n(x)|^2d\mu(y)\,d\mu(x)\\
&=\frac{C}{\varepsilon^{d_w}}\sum_j \sum_{i\colon B_i^{2\varepsilon}\cap B_j^{2\varepsilon}\neq \emptyset}\int_{B_j^\varepsilon}\fint_{B(x,6\varepsilon)}|f_n(y)-f_n(x)|^2d\mu(y)\,d\mu(x)\\
&\leq \frac{C}{\varepsilon^{d_w}}\sum_j\int_{B_j^\varepsilon}\fint_{B(x,6\varepsilon)}|f_n(y)-f_n(x)|^2d\mu(y)\,d\mu(x)\\
&\leq \frac{C}{\varepsilon^{d_w}}\int_X\fint_{B(x,6\varepsilon)}|f_n(y)-f_n(x)|^2d\mu(y)\,d\mu(x).
\end{align*}
which is~\eqref{E:KS_cond_dw_help_b}. The proof is complete.
\end{proof}

\begin{prop}\label{P:KS_as_Gamma_dw}
There exists a strongly local and regular Dirichlet form $(\mathcal{E},KS^{d_w/2,2} (X))$ on $L^2(X,\mu)$ such that for every $f \in KS^{d_w/2,2} (X)$
\begin{equation}\label{E:KS_comp_Gamma_dw}
C_1  \sup_{r >0} E_{d_w/2,X}(f,r) \le \mathcal{E}(f,f) \le C_2 \liminf_{r\to 0^+} E_{d_w/2,X}(f,r).
\end{equation}
In addition, for a suitable sequence $\{r_n\}_{n\geq 1}$  converging to zero
\begin{equation}\label{E:KS_as_Gamma_dw}
    \mathcal{E}(f,f)=\Gamma{-}\!\!\lim_{n\to\infty}E_{d_w/2,X}(f,r_n).
\end{equation}
\end{prop}

\begin{proof}
Analogous to Theorem~\ref{T:KS_as_Gamma}, in view of Lemma~\ref{L:Kumagai-Sturm_condition_dw} the conditions of~\cite[Theorem 2.1]{KS05} are satisfied.
Thus, there exists a strongly local Dirichlet form fulfilling~\eqref{E:KS_comp_Gamma_dw} and~\eqref{E:KS_as_Gamma_dw}. 
The regularity of the form is a  consequence of Assumption \ref{SS:non-strictly}, which allows to take $KS^{d_w/2,2} (X)\cap C(X)$ as core. Indeed, let $f \in KS^{d_w/2,2} (X)$. Let $\{f_n\}\subset L^2(X,\mu)$ so that $f_n\to f$ in $L^2(X,\mu)$ and $E_{d_w/2,X}(f,r_n) \to \mathcal{E}(f,f)$. Fix $\varepsilon>0$ and consider the $\varepsilon$-covering $\{B_i^{\varepsilon}\}_{i\geq 1}$ and the partition of unity $\{\varphi_i^\varepsilon\}_{i}$ from the controlled cutoff condition. As before, set
\begin{equation*}
f_{n,\varepsilon}:=\sum_if_{n,B_i^\varepsilon}\varphi_i^\varepsilon.
\end{equation*}
Note that from \eqref{E:KS_cond_dw_help} for any $r>0$, 
\begin{equation*}
\frac{1}{r^{d_w}}\int_X\fint_{B(x,r)}|f_{n,r_n/6}(x)-f_{n,r_n/6}(y)|^2d\mu(y)\,d\mu(x) \le \frac{C}{r_n^{d_w}} \int_X \fint_{B(x,r_n)}|f_n(y)-f_n(x)|^2d\mu(y)\, d\mu(x)
\end{equation*}
and that $f_{n,r_n/6}$ converges to $f$ in $L^2(X,\mu)$. Therefore the sequence $f_{n,r_n/6}$ is bounded in $KS^{d_w/2,2}(X)$ for the norm $\|f\|_{L^2(X,\mu)}+ \sup_{r>0} E_{d_w/2,X}(f,r)$. Since the space $KS^{d_w/2,2}(X)$ is reflexive, see \cite{Bau22}, we deduce that there exists a subsequence $f_{n_k,r_{n_k}/6}$ that converges weakly to $f$. Mazur's lemma implies  then that there is a convex combination of the $f_{n_k,r_{n_k}/6}$'s that converges strongly to $f$. This convex combination is in   $KS^{d_w/2,2} (X)\cap C(X)$, therefore the space  $KS^{d_w/2,2}(X)\cap C(X)$ is dense in $KS^{d_w/2,2}(X)$ for the norm $\|f\|_{L^2(X,\mu)}+ \sup_{r>0} E_{d_w/2,X}(f,r)$. 
Finally, for $f \in C(X)$, the sequence 
\[
f_{\varepsilon}:=\sum_i f(x_i) \varphi_i^\varepsilon \in KS^{d_w/2,2} (X)\cap C(X)
\]
does converge uniformly  when $\varepsilon \to 0$ to $f$ where $x_i$ is the center of $B_i^\varepsilon$. Thus the space  $KS^{d_w/2,2}(X)\cap C(X)$ is dense  in $C(X)$  for the supremum norm.
\end{proof}


\subsection{Mosco convergence}
As in the strictly local case, the Mosco convergence is obtained by showing that the sequence of forms $\{E_{d_w/2,X}(\cdot,r_n)\}_{n\geq 0}$ is asymptotically compact. The analogue of Lemma~\ref{L:asymp_cpt_cond} is now a consequence of the Rellich-Kondrachov Theorem proved in Section~\ref{S:Rellich-Kondrachov}.
    
\begin{lem}\label{L:asymp_cpt_cond_dw}
Let $\{\varepsilon_n\}_{n\geq 1}$ with $\varepsilon_n\to 0$. Any sequence $\{f_n\}_{n\geq 1}\subset L^2(X,\mu)$ such that 
\begin{equation}\label{E:asymp_cpt_cond_dw}
\liminf_{n\to\infty} (E_{d_w/2,X}(f_n,\varepsilon_n)+\|f_n\|_{L^2(X,\mu)}^2)<+\infty
\end{equation}
has a subsequence that converges strongly in $L^2(X,\mu)$.
\end{lem}

\begin{proof}
After possibly extracting a subsequence we might assume
\[
\sup_{n\ge 1} (E_{d_w/2,X}(f_n,\varepsilon_n)+\|f_n\|_{L^2(X,\mu)}^2)<+\infty
\]
Consider the sequence $\{f_{n,\varepsilon_n/6}\}_{n\geq 1}$, where $f_{n,\varepsilon_n/6}$ is defined as in~\eqref{E:def_f_eps_dw}. This sequence is easily seen to be bounded in $L^2$ since $f_n$ is. 
As in~\eqref{E:KS_cond_dw_help_b} from the proof of Lemma~\ref{L:Kumagai-Sturm_condition_dw}, 
\begin{equation*}
    \limsup_{r\to 0^+}E_{d_w/2,X}(f_{n,\varepsilon_{n}/6},r)\leq \frac{C}{\varepsilon_n^{d_w}}\int_X\fint_{B(x,\varepsilon_{n})}|f_{n}(y)-f_{n}(x)|^2d\mu(y)\,d\mu(x)= CE_{d_w/2,X}(f_{n},\varepsilon_{n})
\end{equation*}
holds uniformly on $n$, whence
\begin{equation*}
    \liminf_{r\to 0^+}E_{d_w/2,X}(f_{n,\varepsilon_{n}/6},r) \leq C\sup_{n \ge 1}E_{d_w/2,X}(f_{n},\varepsilon_{n}).
\end{equation*}
From  Corollary~\ref{C:Rellich-Kondrachev-L2} one can find a  subsequence $f_{n_k,\varepsilon_{n_k}/6}$ that converges in $L^2$. Let us call $f\in L^2(X,\mu)$ that limit. Following the same arguments as in~\eqref{E:subseq_conv_Gaussian}  the subsequence $\{f_{n_k}\}_{k\geq 1}$ converges to $f$ because
\begin{equation*}
    \|f-f_{n_k}\|_{L^2(X,\mu)}\leq \|f-f_{n_k,\varepsilon_{n_k/6}}\|_{L^2(X,\mu)}+C\varepsilon_{n_k}^{d_w}E_{d_w/2,X}(f_{n_k},\varepsilon_{n_k})\xrightarrow{k\to\infty}0.
\end{equation*}
\end{proof}

\begin{cor}\label{C:Mosco_limit_dw}
The Dirichlet form $(\mathcal{E},KS^{1,2} (X))$ from Proposition~\ref{P:KS_as_Gamma_dw} is also the Mosco limit of the sequence in~\eqref{E:KS_as_Gamma_dw}.
\end{cor}

\subsection{2-Poincar\'e inequality with respect to energy measures}
The 2-Poincar\'e inequality with respect to an energy measure $d\Gamma(f,f)$ is obtained in a similar fashion as in the strictly local case. Substituting the exponent $2$ by $d_w$, one can prove as in Proposition~\ref{P:limsup_vs_Gamma} that there exists $C>0$ and $\lambda>1$ so that
\begin{equation}\label{E:limsup_vs_Gamma_dw}
\limsup_{r\to 0^+}\frac{1}{r^{d_w}}\int_{B(x,R)}\fint_{B(z,r)}|f(z)-f(y)|^2d\mu(y)\,d\mu(z)\leq C\int_{B(x,\lambda R)}d\Gamma(f,f)
\end{equation}
for any $f\in KS^{d_w/2,2}(X)\cap C(X)$, $x\in X$ and $R>0$. The proof of the 2-Poincar\'e inequality with respect to $\Gamma$ simplifies in this case because the standing 2-Poincar\'e inequality from Assumption~\ref{A:2PI_KS_dw} holds for all functions in $ KS^{d_w/2,2}(X)$.
\begin{theorem}\label{T:2PI_Gamma_dw}
There exist $C>0$ and $\Lambda>1$ such that 
\begin{equation}\label{E:2PI_Gamma_dw}
\int_{B(x,R)}|f(y)-f_{B(x,R)}|^2d\mu(z)\leq CR^{d_w}\int_{B(x,\Lambda R)}d\Gamma(f,f)
\end{equation}
for any $f\in KS^{d_w/2,2}(X)$, $x\in X$ and $R>0$.
\end{theorem}

\begin{proof}
    By virtue of the 2-Poincar\'e inequality~\eqref{E:2PI_KS_dw} and~\eqref{E:limsup_vs_Gamma_dw} here exists $C>0$ and $\lambda,\Lambda>1$ such that 
    \begin{align*}
        \int_{B(x,R)}|f(y)-f_{B(x,R)}|^2d\mu(z)&\leq C R^{d_w} \liminf_{r\to 0^+} E_{d_w/2,B(x,\lambda R)}(f,r)\\
        &\leq C R^{d_w} \limsup_{r\to 0^+} E_{d_w/2,B(x,\lambda R)}(f,r)\\
        &\leq CR^{d_w} \int_{B(x,\Lambda R)}d\Gamma(f,f).
    \end{align*}
\end{proof}
%

\subsection{Non-strict locality}
So far it may seem that the case $d_w=2$ generalizes to $d_w>2$ with the appropriate modifications. However, a crucial difference displays when it comes to the geometry of the underlying space that is intrinsic to the Dirichlet form $(\mathcal{E},KS^{1,2}(X))$. While for $d_w=2$ the intrinsic distance $d_{\mathcal{E}}$ generated the same topology as the underlying distance $d$, that property usually fails when $d_w>2$.

\medskip

To make full use of the results available in the literature, we will replace Assumption~\ref{A:CCC} with the following capacity condition from~\cite[(1.20)]{GHL15}: for any $f\in KS^{1,2}(X)\cap L^\infty(X,\mu)$ and any balls $B(x,R)$, $B(x,R+r)$ there exists a cutoff function $\varphi$ with $\varphi\equiv 1$ in $B(x,R)$ and $\supp\varphi\subset B(x,R+r)$ such that
\begin{equation}\label{A:CS_dw}
    \int_Xf^2d\Gamma(\varphi,\varphi)\leq c_1\int_{B(x,R+r)\setminus B(x,R)}d\Gamma(f,f)+\frac{c_2}{r^{d_w}}\int_{B(x,R+r)\setminus B(x,R)}f^2d\mu.\tag{$\rm CS_{d_w}$}
\end{equation}

Assuming in addition a chain condition, see e.g.~\cite[Definition 2.10]{KM20}, the non-strict locality of $(\mathcal{E},KS^{1,2}(X))$ follows from~\cite[Theorem 2.13]{KM20}.
\begin{cor}
    Replacing Assumption~\ref{A:CCC} with~\eqref{A:CS_dw} and adding the chain condition, the energy measure associated with any non constant $f\in KS^{1,2}(X)$ is singular with respect to the underlying measure $\mu$.
\end{cor}
\begin{proof}
    By virtue of~\cite[(1.21)]{GHL15}, assumption~\eqref{A:CS_dw} is equivalent to the cutoff Sobolev condition from~\cite[Theorem 2.13]{KM20} with $\Psi(r)=r^{d_w}$. Together with Assumption~\eqref{A:2PI_KS_dw},~\eqref{A:VD} and the chain condition we may apply~\cite[Theorem 2.13]{KM20} with $\Psi(r)=r^{d_w}$. Since $d_w>2$,
    \[
    \liminf_{\lambda\to\infty}\liminf \frac{\lambda^2\Psi(r/\lambda)}{\Psi(r)}=\liminf_{\lambda\to\infty}\lambda^{2-d_w}=0
    \]
    the singularity of the energy measures $\Gamma(f,f)$ for all $f\in KS^{1,2}(X)$ follow.
\end{proof}

\begin{remark}\label{R:CS_dw_vs_CCC}
The capacity condition~\eqref{A:CS_dw} implies the controlled cutoff condition in Assumption~\ref{A:CCC}, c.f. Subsection~\ref{SS:converse_main}. It is still an open question to determine their precise relation. 
\end{remark}

\subsection{A converse to Theorem~\ref{T:main_dw} }\label{SS:converse_main}

We conclude the paper with the following result that may be regarded as a converse to the main Theorem~\ref{T:main_dw}.

\begin{thm}\label{converse}
Let $(X,d,\mu)$ be a compact metric measure space equipped with a regular Dirichlet form $(\mathcal{E},\mathcal{F})$ whose associated heat semigroup admits  a continuous heat kernel with sub-Gaussian estimates 
\begin{equation}\label{E:subG_HKE}
    p_t(x,y)\simeq \frac{c_1}{\mu(B(x,t^{1/d_w}))}\exp \bigg(\!-c_2\Big(\frac{d(x,y)}{t^{1/d_w}}\Big)^{\frac{d_w}{d_w-1}}\bigg),
    \end{equation}
    where $d_w>2$. Then, $\mathcal{F}=KS^{d_w/2,2}(X)$ and  Assumptions  \ref{A:2PI_KS_dw} and \ref{A:CCC} are satisfied.
 \end{thm}
 
 The fact that the sub-Gaussian estimates \eqref{E:subG_HKE} imply  $\mathcal{F}=KS^{d_w/2,2}(X)$ and 
 \[
 \sup_{r >0} E_{d_w/2,X}(f,r) \simeq \mathcal{E}(f,f)
 \]
 is well-known, see for instance \cite{BV2}. On the other hand, see e.g.~\cite[Theorem 2.8]{KM20}, volume doubling and~\eqref{E:subG_HKE} imply the 2-Poincar\'e inequality with respect to the energy measure~\eqref{E:2PI_Gamma_dw} and the capacitary estimate~\eqref{A:CS_dw}. To conclude the proof of Theorem \ref{converse}, it remains to show that~\eqref{E:subG_HKE}  also imply the 2-Poincar\'e inequality with respect to the Korevaar-Schoen energy from Assumption~\eqref{A:2PI_KS_dw} and the controlled cut off condition from Assumption~\eqref{A:CCC} respectively, both possibly with different constants.

\begin{prop}\label{P:2PI_Gamma_KS}
 The sub-Gaussian estimates \eqref{E:subG_HKE} imply the existence of constants $C>0$ and $\Lambda >1$ such that for every $f \in KS^{d_w/2,2}(X)$, $x \in X$ and $R>0$
 \[
  \int_{B(x, R)}d\Gamma(f,f) \le C  \liminf_{r\to 0^+} E_{d_w/2,B(x,\Lambda R)}(f,r).
 \]
 In particular, \eqref{E:subG_HKE} implies Assumption~\eqref{A:2PI_KS_dw}.
\end{prop}
\begin{proof}
Let $\phi_x\in C(X)$, $0 \le \phi_x \le 1$ be a cut off function with $\phi\equiv 1$ in $B(x,R)$ and $\supp\phi\subseteq B(x,\Lambda R)$. Then,
\begin{equation}\label{E:2PI_Gamma_KS_01}
    \int_{B(x, R)}d\Gamma(f,f)\leq \int_X\phi_xd\Gamma(f,f)=\lim_{t\to 0^+}\frac{1}{2t}\int_X\int_X\phi_x(z)(f(z)-f(y))^2p_t(z,y)\,d\mu(y)\,d\mu(z).
\end{equation}
Fix $\delta>0$ to be specifically chosen later and set
\[
\psi_x(t):=\frac{1}{t}\int_X\int_X\phi_x(z)(f(z)-f(y))^2p_t(z,y)\,d\mu(y)\,d\mu(z).
\]
On the one hand, the upper heat kernel bound $p_t(z,y)\leq C\mu(B(z,t^{1/d_w}))^{-1}$ and the volume doubling property yield
\begin{align}
&\frac{1}{t}\int_X\int_{B(z,\delta t^{1/d_w})}\phi_x(z)(f(z)-f(y))^2p_t(z,y)\,d\mu(y)\,d\mu(z)\notag\\
&\leq \frac{c}{t}\int_X\fint_{B(z,\delta t^{1/d_w})}\phi_x(z)(f(z)-f(y))^2d\mu(y)\,d\mu(z)=:\Phi_x(t).\label{E:2PI_Gamma_KS_02}
\end{align}
On the other hand, the lower heat kernel bound implies
\begin{align}
 &\frac{1}{t}\int_X\int_{X\setminus B(z,\delta t^{1/d_w})}\phi_x(z)(f(z)-f(y))^2p_t(z,y)\,d\mu(y)\,d\mu(z)\notag\\
&\leq \frac{c}{t}e^{-c'\delta^{d_w/(d_w-1)}}\int_X\int_X\phi_x(z)(f(z)-f(y))^2p_{ct}(z,y)\,d\mu(y)\,d\mu(z)=:A_\delta\psi_x(ct),\label{E:2PI_Gamma_KS_03}
\end{align}
so that $\psi_x(t)\leq \Phi_x(t)+A_\delta\psi_x(ct)$. Choosing $\delta>0$ large enough to make $A_\delta<1/2$, it follows from~\eqref{E:2PI_Gamma_KS_02} and~\eqref{E:2PI_Gamma_KS_03} that $\lim_{t\to 0^+}\psi_x(t)\leq C\liminf_{t\to 0^+}\Phi_x(t)$. Together with~\eqref{E:2PI_Gamma_KS_01}, the fact that $\supp\phi_x\subset B(x,\Lambda R)$ and Lemma~\ref{L:4KS_cond_dw},
\begin{align}
    \int_{B(x, R)}d\Gamma(f,f)&\leq C\liminf_{t\to 0^+}\frac{1}{t}\int_X\fint_{B(x,\delta t^{1/d_w})}\phi_x(z)(f(z)-f(y))^2d\mu(y)\,d\mu(z)\notag\\
    &\le C\liminf_{t\to 0^+}\frac{1}{t}\int_{B(x,\Lambda R)}\fint_{B(z,\delta t^{1/d_w})} (f(z)-f(y))^2d\mu(y)\,d\mu(z)\notag\\
    &=C\liminf_{r\to 0^+}\frac{1}{r^{d_w}}\int_{B(x,\Lambda R)}\fint_{B(z,r)} (f(z)-f(y))^2d\mu(y)\,d\mu(z)\notag\\
    &=C\liminf_{r\to 0^+}E_{d_w/2,B(x,\Lambda R)}(f,r)\notag.
    \end{align}
\end{proof}
We finish this section with the proof of the implication announced in Remark~\ref{R:CS_dw_vs_CCC}.

\begin{prop}\label{P:CSdw_CCC}
  The cutoff Sobolev condition ~\eqref{A:CS_dw} implies Assumption~\eqref{A:CCC} 
\end{prop}
\begin{proof}

By standard covering arguments, the cutoff Sobolev condition ~\eqref{A:CS_dw} with $f=1$ implies that  for every $\eps >0$ there exists a covering $\{B_i^\eps=B(x_i,\eps)\}_i$ of $X$, so that $\{B_i^{5\eps}\}_i$ has the bounded overlap property (uniformly in $\eps$) and an associated family of functions $\pip_i^\eps$ satisfying
\begin{itemize}
\item $\pip_i^\eps \in \mathcal{F}=KS^{d_w/2,2}(X)$;
\item $0\le \pip_i^\eps\le 1$ on $X$;
\item $\sum_i\pip_i^\eps=1$ on $X$; 
\item $\pip_i^\eps=0$ in $X\setminus B_i^{2\eps}$;
\item $\mathcal{E}(\pip_i^\eps,\pip_i^\eps) \le C\varepsilon^{-d_w} \mu(B_i^{\eps})$.
\end{itemize}
Since, $ \sup_{r >0} E_{d_w/2,X}(\pip_i^\eps,r)\le \mathcal{E}(\pip_i^\eps,\pip_i^\eps)$, the claim follows immediately.
%
\end{proof}
\bibliographystyle{amsplain}
\bibliography{Mosco_convergence.bib}

\bigskip

\noindent
\textbf{Patricia Alonso Ruiz}\\
Department of Mathematics\\ 
Texas A{\&}M University\\ 
Mailstop 3368\\ 
College Station TX 77848\\
\texttt{paruiz@tamu.edu}

\medskip

\noindent
\textbf{Fabrice Baudoin}\\
Department of Mathematics\\
University of Connecticut\\
Storrs CT 06268\\
\texttt{fabrice.baudoin@uconn.edu}
\end{document}